\documentclass[reqno,11pt]{amsart}
\usepackage{amsmath}
\usepackage{amsxtra}
\usepackage{amscd}
\usepackage{amsthm}
\usepackage{amsfonts}
\usepackage{amssymb}
\usepackage{eucal}
\usepackage[matrix,arrow,curve]{xy}

\allowdisplaybreaks

 \newtheorem{dfn}{Definition}[section]
 \newtheorem{thm}[dfn]{Theorem}
 \newtheorem{prp}[dfn]{Proposition}
 \newtheorem{lem}[dfn]{Lemma}

 \newtheorem{rmk}[dfn]{Remark}
  \newtheorem{con}[dfn]{Conjecture}

\newcommand{\Nk}{{\mathsf N}}

\newcommand{\ignore}[1]{}

\title{Affine Screening Operators, Affine Laumon Spaces, and  Conjectures Concerning Non-Stationary Ruijsenaars Functions}
\author{J.~Shiraishi}
\address{J. Shiraishi: Graduate School of Mathematical Sciences, University of Tokyo, Komaba, Tokyo 153-8914, Japan}
\email{shiraish@ms.u-tokyo.ac.jp}

\begin{document}

\dedicatory{Dedicated to the memory of Professor Tohru Eguchi}

\begin{abstract}
Based on the screened vertex operators associated with the affine screening operators, 
we introduce the formal power series $f^{\widehat{\mathfrak gl}_N}(x,p|s,\kappa|q,t)$ which we call the {\it non-stationary Ruijsenaars function}. 
We identify it with the generating function for the Euler characteristics of the affine Laumon spaces.
When the parameters $s$ and $\kappa$ are suitably chosen, the limit $t\rightarrow q$ of 
$f^{\widehat{\mathfrak gl}_N}(x,p|s,\kappa|q,q/t)$ gives us the dominant integrable characters of $\widehat{\mathfrak sl}_N$ multiplied by 
$1/(p^N;p^N)_\infty$ ({\it i.e.} the $\widehat{\mathfrak gl}_1$ character). Several conjectures are presented for 
$f^{\widehat{\mathfrak gl}_N}(x,p|s,\kappa|q,t)$, including the bispectral and the Poincar\'e dualities, and the evaluation formula.
The main conjecture asserts that (i) one can normalize  $f^{\widehat{\mathfrak gl}_N}(x,p|s,\kappa|q,t)$ in such a way that 
the limit $\kappa\rightarrow 1$ exists,  and (ii) the limit $f^{{\rm st.}\,\widehat{\mathfrak gl}_N}(x,p|s|q,t)$ gives us the eigenfunction 
of the elliptic Ruijsenaars operator. 
The non-stationary affine $q$-difference Toda operator ${\mathcal T}^{\widehat{\mathfrak gl}_N}(\kappa)$ is introduced,
which comes as an outcome of the study of the Poincar\'e duality conjecture in the affine Toda limit $t\rightarrow 0$.
The main conjecture is examined also in the limiting cases of the affine $q$-difference Toda ($t\rightarrow 0$), and the elliptic 
Calogero-Sutherland ($q,t\rightarrow 1$) equations.
\end{abstract}
\maketitle

\section{Introduction}
In this article, we  introduce an
affine analogue $f^{\widehat{\mathfrak gl}_N}(x,p|s,\kappa|q,t)$ of the asymptotically free eigenfunction  $f^{{\mathfrak gl}_N}(x|s|q,t)$ \cite{S,NS,BFS} for the 
Macdonald operator \cite{M}. We call $f^{\widehat{\mathfrak gl}_N}(x,p|s,\kappa|q,t)$ the {\it non-stationary Ruijsenaars function}. 
We derive it from a construction based on the algebra of the affine (or toroidal) screening operators \cite{FKSW1,FKSW2,KS}. (See Theorem \ref{affine-screening} in 
\S \ref{affine-screening-algebra}.) 
Then, comparing the characters, we identify $f^{\widehat{\mathfrak gl}_N}(x,p|s,\kappa|q,1/t)$
with the Euler characteristics of the affine Laumon spaces \cite{FFNR}. (See Theorem \ref{affine-Laumon} 
and Proposition \ref{affine-Laumon-2} in \S \ref{affine-Laumon-space}.)
Most probably, when the parameters are chosen as in Definition \ref{dominant-integrable-weights} below
(corresponding to the dominant integrable representations of $\widehat{\mathfrak{sl}}_N$), 
$f^{\widehat{\mathfrak gl}_N}(x,p|s,\kappa|q,q/t)$ coincides with 
Etingof and Kirillov Jr.'s affine Macdonald polynomial \cite{EK}, up to some normalization factor. 
Based on the same philosophy as in the work of Atai and Langmann for the non-stationary Heun and Lam\'e equations \cite{AL},
we present several conjectures which support the idea that the 
 $f^{\widehat{\mathfrak gl}_N}(x,p|s,\kappa|q,t)$  could be applied for 
the eigenvalue problems associated with the elliptic Ruijsenaars operator \cite{R}, and whose particular degeneration limits including 
the $q$-difference affine Toda system, elliptic Calogero-Sutherland system.

Let $N\in \mathbb{ Z}_{\geq 2}$. Introduce the collections of independent indeterminates
\begin{align}
&(x,p)=(x_1,x_2,\ldots,x_N, p),\qquad 
(s,\kappa)=(s_1,s_2,\ldots,s_N,\kappa). \nonumber
\end{align} 
Extend the indices of $x$ and $s$ to $\mathbb{Z}$, assuming 
the cyclic identifications $x_{i+N}=x_i$ and $s_{i+N}=s_i$. 
Let $\omega$ be the permutation acting on $(x,p)$ and $(s,\kappa)$ by
$\omega x_i=x_{i+1}, \omega p=p, \omega s_i=s_{i+1}, \omega \kappa=\kappa$.
A sequence $\lambda=(\lambda_1,\lambda_2,\ldots)$ of non-increasing non-negative integers 
with finitely many positive parts
is called a partition, {\it i.e.} $\lambda_i \in \mathbb{Z}_{\geq 0}$, 
$\lambda_1\geq\lambda_2\geq\cdots$, and $|\lambda|:=\sum_i\lambda_i<\infty$. 
Let ${\mathsf P}$ be the set of all partitions. Transposition of $\lambda$ is denoted by $\lambda'$. 
We use the standard notation for the shifted products as in (\ref{shifted-products}) below, see \cite{GR} as for the detail.

\begin{dfn}
For $k\in \mathbb{ Z}/N\mathbb{Z}$, and $\lambda,\mu \in {\mathsf P}$, set
\begin{align*}
&\Nk^{(k|N)}_{\lambda,\mu}(u|q,\kappa)=
\Nk^{(k)}_{\lambda,\mu}(u|q,\kappa)\\
=&
 \prod_{j\geq i\geq 1 \atop j-i \equiv k \,\,({\rm mod}\,N)}
(u q^{-\mu_i+\lambda_{j+1}} \kappa^{-i+j};q)_{\lambda_j-\lambda_{j+1}}
\cdot
\prod_{\beta\geq \alpha \geq 1  \atop \beta-\alpha \equiv -k-1 \,\,({\rm mod}\,N)}
(u q^{\lambda_{\alpha}-\mu_\beta} \kappa^{\alpha-\beta-1};q)_{\mu_{\beta}-\mu_{\beta+1}}. \nonumber
\end{align*}
\end{dfn}

Note that the ordinary $K$-theoretic Nekrasov factor \cite{Nek} reads 
\begin{align}
\Nk_{\lambda,\mu}(u|q,\kappa)
&=
\prod_{(i,j)\in \lambda}(1-u q^{-\mu_i+j-1}\kappa^{\lambda'_j-i}) \cdot 
\prod_{(k,l)\in \mu}(1-u q^{\lambda_k-l}\kappa^{-\mu'_l+k-1}),\nonumber 
\end{align}
or equivalently 
\begin{align*}
&
 \Nk_{\lambda,\mu}(u|q,\kappa)=
 \prod_{j\geq i\geq 1}
(u q^{-\mu_i+\lambda_{j+1}} \kappa^{-i+j};q)_{\lambda_j-\lambda_{j+1}}\cdot
\prod_{\beta\geq \alpha \geq 1}
(u q^{\lambda_{\alpha}-\mu_\beta} \kappa^{\alpha-\beta-1};q)_{\mu_{\beta}-\mu_{\beta+1}}.
\end{align*}
We have the factorization $ \Nk_{\lambda,\mu}(u|q,\kappa)=\prod_{k=1}^N \Nk^{(k|N)}_{\lambda,\mu}(u|q,\kappa)$.

\begin{dfn}
Let 
$f^{\widehat{\mathfrak gl}_N}(x,p|s,\kappa|q,t)$ be the formal power series
	\begin{align}
&f^{\widehat{\mathfrak gl}_N}(x,p|s,\kappa|q,t)
\in \mathbb{Q}(s,\kappa,q,t)[[px_2/x_1,\ldots,px_{N}/x_{N-1},p x_1/x_{N}]],\nonumber\\
&f^{\widehat{\mathfrak gl}_N}(x,p|s,\kappa|q,t)
=
\sum_{\lambda^{(1)},\ldots,\lambda^{(N)}\in {\mathsf P}}
\prod_{i,j=1}^N
{\Nk^{(j-i|N)}_{\lambda^{(i)},\lambda^{(j)}} (ts_j/s_i|q,\kappa) \over \Nk^{(j-i|N)}_{\lambda^{(i)},\lambda^{(j)}} (s_j/s_i|q,\kappa)}
\cdot \prod_{\beta=1}^N\prod_{\alpha\geq 1} ( p x_{\alpha+\beta}/tx_{\alpha+\beta-1})^{\lambda^{(\beta)}_\alpha}.\nonumber
\end{align}
We call $f^{\widehat{\mathfrak gl}_N}(x,p|s,\kappa|q,t)$ the non-stationary Ruijsenaars function.
\end{dfn}
Note that we have 
$\omega f^{\widehat{\mathfrak gl}_N}(x,p|s,\kappa|q,t)=f^{\widehat{\mathfrak gl}_N}(x,p|s,\kappa|q,t)$.
In \S \ref{affine-screening-algebra}, a derivation is presented for the series $ f^{\widehat{\mathfrak gl}_N}(x,p|s,\kappa|q,t)$ 
as a matrix element of a composition of certain screened vertex operators associated with the affine (or toroidal) screening operators found in \cite{FKSW1,FKSW2,KS}.
See Theorem \ref{affine-screening}. For the moment, the role of the deformed $W$-algebras associated with the affine screening operators remains unclear. 

A simple calculation using the $q$-binomial formula \cite{GR} gives us the following factorization formula. See Remark \ref{kappa=0-Mac} below.
\begin{prp}\label{kappa=0}
Setting $\kappa=0$, we have
\begin{align}
f^{\widehat{\mathfrak gl}_N}(x,p|s,0|q,t)
=&
\prod_{1\leq i<j\leq N}
{(p^{j-i}qx_j/x_i;q,p^N)_\infty \over (p^{j-i}tx_j/x_i;q,p^N)_\infty }
\cdot 
\prod_{1\leq i\leq j\leq N}
{(p^{N-j+i}qx_i/x_j;q,p^N)_\infty \over (p^{N-j+i}tx_i/x_j;q,p^N)_\infty }.\nonumber
\end{align}
\end{prp}

Dividing $f^{\widehat{\mathfrak gl}_N}(x,p|s,\kappa|q,t)$ by $f^{\widehat{\mathfrak gl}_N}(x,p|s,0|q,t)$, 
we introduce the normalized version $\varphi^{\widehat{\mathfrak gl}_N}(x,p|s,\kappa|q,t)$ as follows. 
\begin{dfn}\label{normalized-form}
Let 
$\varphi^{\widehat{\mathfrak gl}_N}(x,p|s,\kappa|q,t)$ be the formal power series
\begin{align}
&\varphi^{\widehat{\mathfrak gl}_N}(x,p|s,\kappa|q,t)
\in \mathbb{Q}(q,t)[[px_2/x_1,\ldots,px_{N}/x_{N-1},p x_1/x_{N},
\kappa s_2/s_1,\ldots,\kappa s_{N}/s_{N-1},\kappa s_1/s_{N}]],\nonumber\\
&\varphi^{\widehat{\mathfrak gl}_N}(x,p|s,\kappa|q,t)\nonumber\\
=&
\prod_{1\leq i<j\leq N}
{(p^{j-i}tx_j/x_i;q,p^N)_\infty \over (p^{j-i}qx_j/x_i;q,p^N)_\infty }
\cdot 
\prod_{1\leq i\leq j\leq N}
{(p^{N-j+i}tx_i/x_j;q,p^N)_\infty \over (p^{N-j+i}qx_i/x_j;q,p^N)_\infty }
\cdot  f^{\widehat{\mathfrak gl}_N}(x,p|s,\kappa|q,t),
\end{align}
where the coefficients 
$\prod_{i,j=1}^N
{\Nk^{(j-i|N)}_{\lambda^{(i)},\lambda^{(j)}} (ts_j/s_i|q,\kappa) / \Nk^{(j-i|N)}_{\lambda^{(i)},\lambda^{(j)}} (s_j/s_i|q,\kappa)}$
in $f^{\widehat{\mathfrak gl}_N}(x,p|s,\kappa|q,t)$ are 
Taylor expanded in $\kappa$ at $\kappa=0$.
\end{dfn}
We have
$\omega \varphi^{\widehat{\mathfrak gl}_N}(x,p|s,\kappa|q,t)=\varphi^{\widehat{\mathfrak gl}_N}(x,p|s,\kappa|q,t)$.

\begin{con}\label{main-con}
We have the duality properties
\begin{align}
&\varphi^{\widehat{\mathfrak gl}_N}(x,p|s,\kappa|q,t)=\varphi^{\widehat{\mathfrak gl}_N}(s,\kappa|x,p|q,t)
 &(\mbox{bispectral duality}),\label{conjecture-bispectral}\\
& 
\varphi^{\widehat{\mathfrak gl}_N}(x,p|s,\kappa|q,t)=
\varphi^{\widehat{\mathfrak gl}_N}(x,p|s,\kappa|q,q/t) &(\mbox{Poincar\'e duality}).\label{conjecture-Poincare}
\end{align}
\end{con}
These bispectral and the Poincar\'e duality conjectures for $\varphi^{\widehat{\mathfrak gl}_N}(x,p|s,\kappa|q,t)$ are 
regarded as affine analogues of the ones for the Macdonald function $\varphi^{\mathfrak{gl}_N}(x|s|q,t)$ 
established in Noumi and the present author's paper \cite{NS}. See Proposition \ref{Macdonald-duality} in \S \ref{Macdonald-function} below.

We show that the series $f^{\widehat{\mathfrak gl}_N}(x,p|s,\kappa|q,1/t)$
is the generating function for the Euler characteristic
${\mathfrak J}_{\underline{d}}(s,\kappa|q,t)
=\sum_{i,j}(-1)^{i+j} t^j [H^i ( \mathcal{P}_{\underline{d}},\Omega_{\mathcal{P}_{\underline{d}}}^j)]$
of the de Rham complex of the affine Laumon space $\mathcal{P}_{\underline{d}}$ studied in \cite{FFNR}. 
See Theorem \ref{affine-Laumon} and Proposition \ref{affine-Laumon-2} in \S \ref{affine-Laumon-space}.
The key for this identification between $f^{\widehat{\mathfrak gl}_N}(x,p|s,\kappa|q,1/t)$
and the geometric object (the Euler characteristics of the affine Laumon spaces),
is a comparison of the combinatorial identities given in Propositions \ref{ch-11} and \ref{ch-22}. 

We remark that in \cite{Neg,Neg2} Braverman's conjecture \cite{B} was proved, 
showing that the generating function of the Chern polynomials of the (affine) Laumon spaces
satisfies the (elliptic) Calogero-Sutherland equation (Theorem {7.1} in \cite{Neg}
and Theorem 1.5 in \cite{Neg2}). 
In view of Conjecture \ref{nonstationary-eCS} below, 
it seems plausible  that the normalized series $\varphi^{\widehat{\mathfrak gl}_N}(x,p|s,\kappa|q,1/t)$
admits a similar interpretation in terms of the Chern polynomials of the affine Laumon spaces, or Hirzebruch-Riemann-Roch theorem.

The Schur polynomials are obtained from the Macdonald polynomials by taking the limit 
$t\rightarrow q$. In the same manner, we have the $\widehat{\mathfrak sl}_N$ dominant integrable characters 
(up to the character of $\widehat{\mathfrak gl}_1$) from 
$f^{\widehat{\mathfrak gl}_N}(x,p|s,\kappa|q,q/t)$ by considering the limit $t\rightarrow q$.
Set $\delta=(N-1,N-2,\ldots, 1,0)$. Here and hereafter,  we use the standard notation as
$t^\delta s=(t^{N-1}s_1,t^{N-2}s_2,\ldots,t s_{N-1},s_N)$.

\begin{dfn}\label{dominant-integrable-weights}
Let $K$ be a nonnegative integer. We call $K$ the level. 
Let $\mu=(\mu_1,\ldots,\mu_N)$ be a partition satisfying the condition 
$K+\mu_N-\mu_1\geq 0$. 
Then set 
\begin{align}
s=(\kappa t)^\delta q^\mu=q^{-K\delta/N+\mu},\qquad \kappa=q^{-K/N} t^{-1}.\label{dominant}
\end{align}
{\it i.e.} for $s$, we set $s_i=q^{-K(N-i)/N+\mu_i}$ ($1\leq i\leq N$).
\end{dfn}

For such $K$ and $\mu$, 
we have the level $K$ dominant integrable weight 
$\Lambda(K,\mu)=(K+\mu_N-\mu_1)\Lambda_0+\sum_{i=1}^{N-1} (\mu_i-\mu_{i+1})\Lambda_i$,
and the dominant integrable representation $L(\Lambda(K,\mu))$ of  $\widehat{\mathfrak sl}_N$, 
where $\Lambda_0,\ldots,\Lambda_{N-1}$ denote the fundamental weights.
Denote by ${\rm ch}^{\widehat{\mathfrak sl}_N}_{L(\Lambda(K,\mu))}$ the character of  $L(\Lambda(K,\mu))$ 
associated with the principal gradation.

\begin{thm}\label{ch-glN}
Let $K,\mu,s,\kappa$ be fixed as in (\ref{dominant}).We have
\begin{align*}
\lim_{t\rightarrow q} x^\mu
f^{\widehat{\mathfrak gl}_N}(x,p|q^{-K\delta/N+\mu},q^{-K/N} t^{-1}|q,q/t)=
{1\over (p^N;p^N)_\infty}\cdot {\rm ch}^{\widehat{\mathfrak sl}_N}_{L(\Lambda(K,\mu))}.
\end{align*}
\end{thm}
Note that the factor $1/ (p^N;p^N)_\infty$ is interpreted as the $\widehat{\mathfrak gl}_1$ character.
A proof of this is given in \S\ref{proofs-ch} based on the affine Gelfand-Tsetlin pattern obtained in  \cite{FFNR},
which we can regard as Tingley's $\widehat{\mathfrak{sl}}_N$-crystal \cite{T}.

\begin{prp}\label{ch-gl1}
Let $K=0,\mu=\emptyset$. Then (\ref{dominant}) means  $s_i=1$ ($1\leq i\leq N$) and $\kappa=t^{-1}$. 
Let $q$ and $t$ be arbitrary. 
In this case, we have
\begin{align*}
f^{\widehat{\mathfrak gl}_N}(x,p|1,\ldots,1,t^{-1}|q,q/t)={1\over (p^N;p^N)_\infty}.
\end{align*}
\end{prp}

Proofs of Theorem \ref{ch-glN} and Proposition \ref{ch-gl1} are given in \S \ref{proofs-ch}

Proposition \ref{ch-gl1} and the Poincar\'e duality in Conjecture \ref{main-con} imply the 
following evaluation formula. 
\begin{con}
Set $x=(1,\ldots,1)$ and $p=1/t$. 
We have the identity (evaluation formula) in 
$\mathbb{Q}(q)[s_2/s_1,s_3/s_2,\ldots,s_1/s_N][[1/t,\kappa]]$ as
\begin{align}
&f^{\widehat{\mathfrak gl}_N}(1,\ldots,1,1/t|s,\kappa|q,q/t)
={1\over (\kappa^N;\kappa^N)_\infty} \Bigl({(q/t;q)_\infty}\Bigr)^N \label{specialization}\\
&\quad\times
\left({(\kappa^Nq/t;q,\kappa^N)_\infty \over (\kappa^Nq;q,\kappa^N)_\infty}\right)^N
\prod_{1\leq i<j\leq N}
{(\kappa^{j-i}qs_j/ts_i;q,\kappa^N)_\infty \over (\kappa^{j-i}qs_j/s_i;q,\kappa^N)_\infty }
{(\kappa^{N-j+i}qs_i/ts_j;q,\kappa^N)_\infty \over (\kappa^{N-j+i}qs_i/s_j;q,\kappa^N)_\infty }.\nonumber
\end{align}
\end{con}

Setting our parameters as in (\ref{dominant}) and letting $t=q^k$ ($k\in \mathbb{Z}_{>0}$), one finds that (\ref{specialization}) looks very close to the 
specialization formula for the affine Macdonald polynomials based on the representation theories of 
quantum affine algebra $U_q(\widehat{\mathfrak sl}_N)$ in \cite{EK}. See Conjecture {11.3} in \cite{EK}, and 
Conjecture {4.3} in  \cite{RSV} as for the improved version.
In view of this correspondence and Conjecture \ref{nonstationary-eCS}  (for elliptic Calogero-Sutherland limit) in \S \ref{nonstationary-eCS-system} below, 
we strongly expect that Etingov and Kirillov Jr.'s affine Macdonald polynomial $\widehat{J}^p_{\widehat{\lambda}}$ 
coincide with $f^{\widehat{\mathfrak gl}_N}(x,p|s,\kappa|q,q/t)$
up to a normalization factor.  

Since our description of $f^{\widehat{\mathfrak gl}_N}(x,p|s,\kappa|q,q/t)$
(via the affine screening operators or the affine Laumon spaces) is quite different from the `trace of intertwiner' construction in \cite{EK}, 
we have not been able to compare two objects, unfortunately. 
It is an intriguing  problem to establish the connection between them.

Now, we turn to the eigenvalue problem associated with the elliptic Ruijsenaars operator \cite{R}, 
from the point of view of the series $f^{\widehat{\mathfrak gl}_N}(x,p|s,\kappa|q,t)$.
We use the multiplicative notation for the elliptic theta function as
$\Theta_p(z)=(z;p)_\infty (p/z;p)_\infty (p;p)_\infty$.  
\begin{dfn}
Let $D_x(p)=D_x(p|q,t)$ denotes the Ruijsenaars operator \cite{R}
\begin{align}
D_x(p)=&\sum_{i=1}^N
\prod_{j\neq i}
{\Theta_{p}(t x_i/x_j) \over \Theta_{p}(x_i/x_j) } T_{q,x_i}, \label{D-Ruijsenaars}
\end{align}
where $T_{q,x_i}$ is the $q$-shift operator $q^{x_i {\partial /\partial x_i}}$. 
\end{dfn}

Naively speaking, we take the ``stationary limit $\kappa\rightarrow 1$ of $f^{\widehat{\mathfrak gl}_N}(x,p|s,\kappa|q,t)$''. 
Such a limit, however, does not exists. It seems that we need to normalize $f^{\widehat{\mathfrak gl}_N}$, 
before taking the limit $\kappa\rightarrow 1$. The simplest way might be to divide $f^{\widehat{\mathfrak gl}_N}$ by its  
constant term in $x$. 

We closely follow the method developed in Atai and Langmann's paper \cite{AL} for the non-stationary Heun and Lam\'e equations. 
Let ${\boldsymbol \lambda}=(\lambda^{(1)},\ldots,\lambda^{(N)})$ be an $N$-tuple of partitions.
Set 
\begin{align*}
|\boldsymbol \lambda|=\sum_{i=1}^N |\lambda^{(i)}|,\qquad 
m_i=m_i({\boldsymbol \lambda})=
\sum_{\beta=1}^N{ \sum_{\alpha\geq 1 \atop \alpha+\beta\equiv i\,\,({\rm mod}\,N)}}\lambda^{(\beta)}_\alpha-\lambda^{(\beta+1)}_\alpha. 
\end{align*}
Then we have
$ \prod_{\beta=1}^N\prod_{\alpha\geq 1} (  px_{\alpha+\beta}/tx_{\alpha+\beta-1})^{\lambda^{(\beta)}_\alpha}=(p/t)^{|\boldsymbol \lambda|}
 \prod_{i=1}^N x_i^{m_i}.
$
Note that when $m_1=\cdots=m_N=0$, we have $|\boldsymbol \lambda|\equiv 0 \,\, ({\rm mod} \,N)$.

\begin{dfn}
Let 
$\alpha(p|s,\kappa|q,t)=\sum_{d\geq 0} p^{Nd} 
\alpha_d(s,\kappa|q,t)$ be the constant term of the series $f^{\widehat{\mathfrak gl}_N}(x,p|s,\kappa|q,t)$ with respect to $x_i$'s.
Namely,
\begin{align}
\alpha(p|s,\kappa|q,t)
=&\sum_{
\lambda^{(1)},\ldots,\lambda^{(N)}\in {\mathsf P}\atop 
m_1=\cdots=m_N=0}
(p/t)^{|\boldsymbol \lambda|}
\prod_{i,j=1}^N
{\Nk^{(j-i|N)}_{\lambda^{(i)},\lambda^{(j)}} (t s_j/s_i |q,\kappa) \over \Nk^{(j-i|N)}_{\lambda^{(i)},\lambda^{(j)}} (s_j/s_i |q,\kappa)}.
\nonumber
\end{align}
\end{dfn}

\begin{con}\label{analyticity-Ruijsenaars}
We have the properties:
\begin{enumerate}
\item
The series $f^{\widehat{\mathfrak gl}_N}(x,p|s,\kappa|q,t)$ is convergent on a certain domain.
With respect to $\kappa$, it is regular on a certain punctured disk 
$\{\kappa \in {\mathbb C}| |\kappa-1| <r,\kappa\neq 1\}$ .
\item
The $f^{\widehat{\mathfrak gl}_N}(x,p|s,\kappa|q,t)$ and $\alpha(p|s,\kappa|q,t)$ are essential singular at $\kappa=1$. 
(The coefficient $\alpha_d(s,\kappa|q,t)$ has a pole of degree $d$ in $\kappa$ at $\kappa=1$.)
\item
The ratio $f^{\widehat{\mathfrak gl}_N}(x,p|s,\kappa|q,t)/\alpha(p|s,\kappa|q,t)$ is regular at $\kappa=1$. 
\end{enumerate}
\end{con}
\begin{dfn}
Assuming Conjecture \ref{analyticity-Ruijsenaars}, set
\begin{align*}
& f^{{\rm st.}\widehat{\mathfrak gl}_N}(x,p|s|q,t)={f^{\widehat{\mathfrak gl}_N}(x,p|s,\kappa|q,t)\over \alpha(p|s,\kappa|q,t)}\Biggl|_{\kappa=1}.
\end{align*}
We call $f^{{\rm st.}\widehat{\mathfrak gl}_N}(x,p|s|q,t)$ the stationary Ruijsenaars function. 
\end{dfn}

Now, we are ready to state our main conjecture.
\begin{con}[Main Conjecture]\label{MAIN}
Let $s=t^\delta q^\lambda$ ($s_i=t^{N-i} q^{\lambda_i}$). Denote by 
$p^{\delta/N} x$ the collection of the shifted coordinates $p^{(N-i)/N} x_i$.
The stationary Ruijsenaars function $x^\lambda f^{{\rm st.}\widehat{\mathfrak gl}_N}(p^{\delta/N} x,p^{1/N}|s|q,q/t)$
is an eigenfunction of the Ruijsenaars operator:
\begin{align*}
&D_x(p)\, x^\lambda f^{{\rm st.}\widehat{\mathfrak gl}_N}(p^{\delta/N} x,p^{1/N}|s|q,q/t)=
\varepsilon(p|s|q,t)\,x^\lambda f^{{\rm st.}\widehat{\mathfrak gl}_N}(p^{\delta/N} x,p^{1/N}|s|q,q/t),\\
&\varepsilon(p|s|q,t)=\sum_{i=1}^Ns_i+ \sum_{d>0} \varepsilon_d(s|q,t) p^{d}.
\end{align*}
\end{con}

To check the conjecture, we first need to formulate the 
eigenvalue problem associated with the Ruijsenaars operator
on the space of formal series $\mathbb{Q}(q,t,s)[[p x_2/x_1,\ldots,p x_N/x_1]]$, clarifying the meaning of the perturbation in $p$. 
Then we can make a systematic check. The detail will be reported elsewhere \cite{LNS} (a joint work with E. Langmann and M. Noumi).

We investigate the implications of  Conjecture \ref{MAIN}, in the Macdonald ($p\rightarrow 0$, in \S 4), the affine $q$-Toda 
($t\rightarrow 0$, in \S 6), and the elliptic Calogero-Sutherland ($q,t\rightarrow 1$, in \S 7) limits.
For the moment, unfortunately, we have not been able to find a non-stationary analogue (a $\kappa$-deformation) of the 
Ruijsenaars operator $D_x(p)$, or $t$-deformation of the non-stationary $q$-affine Toda operator ${\mathcal T}^{\widehat{\mathfrak gl}_N}(\kappa)$
in Definition \ref{NST-op} below, for which $f^{\widehat{\mathfrak gl}_N}(x,p|s,\kappa|q,q/t)$ should give us
the eigenfunction. We strongly expect that Felder and Varchenko's $q$-deformed KZB heat equation provides us with the answer \cite{FV1,FV2,FV3}.

\medskip

The present paper is organized as follows. In Section 2, we introduce the screened vertex operator $\Phi^i(z|x,p)$ 
associated with the algebra of the affine screening operators. In Section 3, we recall some basic facts concerning the affine Laumon spaces, and 
identify $f^{\widehat{\mathfrak gl}_N}(x,p|s,\kappa|q,1/t)$ with the generating function for the Euler characteristics of the affine Laumon spaces.
It is shown that we have the dominant integrable characters from $f^{\widehat{\mathfrak gl}_N}(x,p|s,\kappa|q,q/t)$ by taking 
the limit $t\rightarrow q$ with $s$ and $\kappa$ being chosen as in  Definition \ref{dominant-integrable-weights}.
In Section 4, we study the Macdonald limit $p\rightarrow 0$. Section 5 is devoted to the derivations of the Macdonald operator, 
the elliptic Calogero-Sutherland operator, and the affine $q$-Toda operator. This technical section is necessary 
to have a unified  picture about the family of elliptic integrable systems based on the elliptic Ruijsenaars operator, 
on which we can argue the validity of Conjecture \ref{main-con}. 
In Section 6, we introduce the operator ${\mathcal T}^{\widehat{\mathfrak gl}_N}(q,\widetilde{p},\kappa)$ 
representing the non-stationary affine Toda in the $q$-difference setting. 
Then we explain what does Conjecture \ref{main-con} imply in the Toda limit $t\rightarrow 0$. 
In Section 7, we test Conjecture \ref{main-con} in the elliptic Calogero-Sutherland limit $q,t\rightarrow 1$.

In this paper we state several Propositions, Lemmas etc. without proof; all such results are standard and can be proved by straightforward 
computations using definitions.

We use the standard notation for the $q$-shifted factorials and 
the double infinite products such as
\begin{align}
(u;q)_\infty =\prod_{i=0}^\infty (1-q^i u),\quad (u;q)_n=(u;q)_\infty/(q^nu;q)_\infty,\quad 
(u;q,p)_\infty =\prod_{i,j=0}^\infty (1-q^i p^ju).\label{shifted-products}
\end{align}
As for the detail, see \cite{GR}.

\section*{Acknowledgements}
The present research is
supported by JSPS KAKENHI (Grant Numbers 15K04808 and 16K05186). 
The author thanks V. Pasquier and  S. Ouvry for their worm hospitality and discussion at  the summer school
``Exact methods in low dimensional statistical physics'' (from 25 July to 4 August, 2017, Institut d'\'Etudes Scientifiques de Carg\`ese),
where part of the results in the present paper was obtained. He is grateful to F. Atai, E. Langmann, M. Noumi, B. Feigin and L. Rybnikov for stimulating discussion
and helpful suggestions.

\section{Affine screening operators and $f^{\widehat{\mathfrak gl}_N}(x,p^{1/N}|s,\kappa^{1/N}|q,t)$ }\label{affine-screening-algebra}
\subsection{Heisenberg algebra $\mathfrak{h}_N$}
As for the Heisenberg algebra and the affine (or toroidal) screening operators, 
we basically follow the construction in \cite{FKSW1,FKSW2,KS}.

\begin{dfn}\label{bosons}
Let $N\in \mathbb{Z}_{\geq2}$. 
Let $\mathfrak{h}_N$ be the Heisenberg algebra generated by 
$\beta^0_n,\beta^1_n,\ldots,\beta^{N-1}_n$ ($n\in \mathbb{ Z}_{\neq 0}$) satisfying the
commutation relations
\begin{align*}
& [\beta^i_n,\beta^i_m]=
n {1-t^{n}\over 1-q^{n}}{1-\kappa^{n}q^{n} t^{-n}\over 1-\kappa^{n}}\delta_{n+m,0} \qquad (i=0,\dots,N-1),\\
&
 [\beta^i_n,\beta^j_m]=
 n {1-t^{n}\over 1-q^{n}}{1-q^{n} t^{-n}\over 1-\kappa^{n}}\kappa^{n(-i+j)/N}
 \delta_{n+m,0}\qquad\quad  (0\leq i<j\leq N-1),\\
&
 [\beta^i_n,\beta^j_m]=
  n {1-t^{n}\over 1-q^{n}}{1-q^{n} t^{-n}\over 1-\kappa^{n}}\kappa^{n(N-i+j)/N} \delta_{n+m,0} \qquad (0\leq j<i\leq N-1).
\end{align*}
Let ${\mathcal F}_N$ be the Fock space associated with $\mathfrak{h}_N$, and let $|0\rangle$ be
the Fock vacuum satisfying $\beta^i_n|0\rangle=0$ ($1\leq i\leq N, n>0$). 
Let  $\langle 0|$ be the dual vacuum satisfying $\langle 0|\beta^i_n=0$ ($1\leq i\leq N, n<0$) and $\langle 0|0\rangle=1$.
\end{dfn}
By an abuse of notation, let $\omega$ be the automorphism of $\mathfrak{h}_N$ defined by 
the cyclic permutation $\omega \,\beta^i_n=\beta^{i+1}_n$ ($0\leq i\leq N-2$), and 
$\omega \,\beta^{N-1}_n=\beta^0_n$. 
Note that we have $\omega^N={\rm id}$.

\begin{dfn}
Set $\alpha^i_n=\kappa^{-n/N}\beta^{i-1}_n-\beta^{i}_n$ ($1\leq i\leq N-1$), and 
$\alpha^0_n=\kappa^{-n/N}\beta^{N-1}_n-\beta^0_n$.
\end{dfn}
We have $\omega\, \alpha^i_n=\alpha^{i+1}_n$,
where $\alpha^{N}_n=\alpha^{0}_n$.

\begin{prp} 
For $0\leq i\leq N-1$, 
we have
\begin{align*}
& [\alpha^i_n,\alpha^i_m]=
n (1+q^n t^{-n}){1-t^n\over 1-q^n}\delta_{n+m,0} , \\
& [\alpha^i_n,\alpha^{i+1}_m]=
-n \,\kappa^{n/N}q^n t^{-n}{1-t^n\over 1-q^n}\delta_{n+m,0},\qquad 
 [\alpha^{i}_n,\alpha^{i-1}_m]=
-n\, \kappa^{-n/N}{1-t^n\over 1-q^n}\delta_{n+m,0},
\end{align*}
and $ [\alpha^{i}_n,\alpha^{j}_m]=0$ otherwise. 
\end{prp}
\begin{prp}
For $0\leq i,j\leq N-1$, 
we have 
\begin{align*}
& [\beta^i_n,\alpha^{i+1}_m]=n\, \kappa^{n/N}q^n t^{-n} {1-t^n\over 1-q^n} \delta_{n+m,0},\qquad 
 [\beta^i_n,\alpha^i_m]=-n {1-t^n\over 1-q^n} \delta_{n+m,0},
\end{align*}
and $[\beta^i_n,\alpha^j_m]=0$ otherwise.
\end{prp}

\subsection{Affine screening operators $S_i(z)$}
In this paper we only work out the case $N\geq 3$. The case $N=2$ has to be treated separately, which we omit.
We remark that  the final result given in Theorem \ref{affine-screening} below also applies for the case  $N=2$.
\begin{dfn}
For $0\leq i\leq N-1$, set 
\begin{align*}
&S_i(z)= \,\, :\exp\left(-
\sum_{n\neq 0} {1\over n}\alpha^i_n z^{-n}
\right): \,\,=
\exp\left(
\sum_{n>0} {1\over n}\alpha^i_{-n} z^{n}
\right) 
\exp\left(
-\sum_{n>0} {1\over n}\alpha^i_n z^{-n}
\right).
\end{align*}
Here and hereafter, we use the standard notation for the normal ordered product $:\bullet:$, 
{\it i.e.} we put all the annihilation operators $\beta^i_{n}$ ($0\leq i\leq N-1$, $n>0$)
to the right of the creation ones $\beta^i_{-n}$ ($0\leq i\leq N-1$, $n>0$). 
Note that we have $\omega S_i(z)=S_{i+1}(z)$. 
We call $S_i(z)$'s the affine screening operators. 
\end{dfn}

\begin{prp}
For $0\leq i,j\leq N-1$, we have
\begin{align*}
&S_i(z)S_i(w)=
{( w/z;q)_\infty \over (t w/z;q)_\infty }
{( qw/tz;q)_\infty \over (q w/z;q)_\infty } :S_i(z)S_i(w):,\\
&S_i(z)S_{i+1}(w)=
{(\kappa^{1/N} q w/z;q)_\infty \over (\kappa^{1/N} qw/tz;q)_\infty }
 :S_i(z)S_{i+1}(w):,\\
&S_{i+1}(z)S_i(w)={(\kappa^{-1/N} t w/z;q)_\infty \over (\kappa^{-1/N} w/z;q)_\infty }
 :S_{i+1}(z)S_i(w):,
 \end{align*}
 and $S_{i}(z)S_j(w)=S_j(w)S_{i}(z)=\,\, :S_i(z)S_j(w):$ otherwise.

\end{prp}
\subsection{Vertex operators $\phi_i(z)$}

\begin{dfn}
For $0\leq i\leq N-1$,
set
\begin{align*}
&\phi_i(z)= \,\, :\exp\left(
\sum_{n\neq 0} {1\over n}\beta^i_n z^{-n}
\right):\,\, =
\exp\left(
-\sum_{n> 0} {1\over n}\beta^i_{-n} z^{n}
\right)
\exp\left(
\sum_{n> 0} {1\over n}\beta^i_n z^{-n}
\right).
\end{align*}
Note that we have $\omega \phi_i(z)=\phi_{i+1}(z)$. 
\end{dfn}

\begin{prp}
We have
\begin{align*}
&\phi_i(z)\phi_i(w)=
{(w/z;q,\kappa)_\infty \over (tw/z;q,\kappa)_\infty }
{(\kappa q w/z;q,\kappa)_\infty \over (\kappa qw/tz;q,\kappa)_\infty }:\phi_i(z)\phi_i(w):\qquad (0\leq i\leq N-1),\\
&\phi_i(z)\phi_j(w)=
{(\kappa^{(-i+j)/N} w/z;q,\kappa)_\infty \over (\kappa^{(-i+j)/N} t w/z;q,\kappa)_\infty }
{(\kappa^{(-i+j)/N} q w/z;q,\kappa)_\infty \over (\kappa^{(-i+j)/N} qw/tz;q,\kappa)_\infty }:\phi_i(z)\phi_j(w): \nonumber\\
&\qquad\qquad\qquad\qquad\qquad\qquad\qquad\qquad\qquad\qquad\qquad\qquad
\qquad (0\leq i<j\leq N-1),\\
&\phi_i(z)\phi_j(w)=
{(\kappa^{(-i+j+N)/N} w/z;q,\kappa)_\infty \over (\kappa^{(-i+j+N)/N} t w/z;q,\kappa)_\infty }
{(\kappa^{(-i+j+N)/N} q w/z;q,\kappa)_\infty \over (\kappa^{(-i+j+N)/N} qw/tz;q,\kappa)_\infty }:\phi_i(z)\phi_j(w):\nonumber\\
&\qquad\qquad\qquad\qquad\qquad\qquad\qquad\qquad\qquad\qquad\qquad\qquad
\qquad (0\leq j<i\leq N-1).
\end{align*}

\end{prp}

\begin{prp}\label{OPE}
For $0\leq i,j\leq N-1$ We have
\begin{align*}
& \phi_i(z) S_{i+1}(w)={(\kappa^{1/N}qw/z;q)_\infty \over (\kappa^{1/N}qw/tz;q)_\infty} : \phi_i(z) S_{i+1}(w):,\\
 &  S_{i+1}(w)\phi_i(z)={(\kappa^{-1/N}tz/w;q)_\infty \over (\kappa^{-1/N}z/w;q)_\infty} : \phi_i(z) S_{i+1}(w):,\\
 & \phi_i(z) S_i(w)={(w/z;q)_\infty \over ( tw/z;q)_\infty} : \phi_i(z) S_i(w):,\quad
   S_i(w)\phi_i(z)={(qz/tw;q)_\infty \over (qz/w;q)_\infty} : \phi_i(z) S_i(w):,\\
&  \phi_i(z) S_j(w)=\,\,: \phi_i(z) S_j(w):, \quad   S_j(w)\phi_i(z)=\,\,: \phi_i(z) S_j(w):\qquad (j\neq i,i+1).
 \end{align*} 
\end{prp}

\begin{prp}\label{fusion}
For $0\leq i\leq N-1$, 
we have the fusion properties 
\begin{align*}
\phi_{i}(z)=\,\, :\phi_{i-1}(\kappa^{1/N}z) S_{i}(z) :.
\end{align*}
\end{prp}

\subsection{Screened vertex operators}
We assume that the indices of $S_i(z)$ and $\phi_i(z)$ are extended to $\mathbb{Z}$ assuming the 
cyclic identifications $S_i(z)=S_{i+N}(z)$, $\phi_i(z)=\phi_{i+N}(z)$.
We use the following notation for the ordered products: 
$\displaystyle \prod^{\curvearrowleft}_{1\leq i\leq \ell} A_i:=A_\ell \cdots A_2A_1$.
Let $0\leq i\leq N-1$, and
$\lambda =(\lambda_1,\ldots,\lambda_{\ell})\in {\mathsf P}$,
where $\ell=\ell(\lambda)$ denotes the length of $\lambda$.
Set
\begin{align*}
\phi^{i}_{\lambda}(z)=&
\phi_{i-\ell}(\kappa^{(\ell+1)/N} z)
\prod^{\curvearrowleft}_{1\leq j\leq \ell}S_{i-j+1}(\kappa^{j/ N} q^{\lambda_{j}}z).
\end{align*}
Then we introduce the stabilized version $\Phi^i_{\lambda}(z)$ as follows.

\begin{dfn}
For $1\leq i\leq N-1$ and $\lambda\in  \mathsf{P}$, set 
\begin{align*}
\Phi^i_{\lambda}(z)&=
\left( {(q/t;q)_\infty \over (q;q)_\infty}\right)^{\ell(\lambda)}
 \phi^i_{\lambda}(z).
 \end{align*}
\end{dfn}
Thanks to the properties in Propositions \ref{OPE} and \ref{fusion}, the  operators $\Phi^{i}_{\lambda}(z)$
are consistently defined for all $\lambda \in  \mathsf{P}$.  

Let $(x,p)=(x_1,\ldots,x_N,p)$ be a collection of parameters. 
Extend it to $x=(x_i)_{i\in \mathbb{Z}}$ assuming the cyclic identification $x_i=x_{i+N}$.
\begin{dfn}\label{SVO}
Define the screened vertex operator $\Phi^i(z|x,p)$ by the infinite series
\begin{align*}
\Phi^i(z|x,p)=\sum_{\lambda\in \mathsf{P}}\Phi^i_{\lambda}(z)
\prod_{k\geq 1} (p^{1/N} x_{N-i+k}/x_{N-i+k-1})^{\lambda_k}.
\end{align*}
\end{dfn}

Set $\omega^i=(\omega^i_1,\ldots,\omega^i_N)=(\overbrace{0,\ldots,0}^{i\,\,{\rm times}},1,\ldots,1)\in \mathbb{Z}^N$ for $1\leq i\leq N$. 
Write $t^{\omega^i} \!x=(t^{\omega^i_1} x_1,\ldots,t^{\omega^i_N} x_N )$ for simplicity.  
Let $(s,\kappa)=(s_1,\ldots,s_N,\kappa)$ be another collection of parameters.
\begin{thm}\label{affine-screening}
Let $N\in \mathbb{Z}_{\geq 2}$.
We have
\begin{align*}
&\langle 0| \Phi^0(s_1|t^{\omega^{N}}\!\!x,p) \Phi^1(s_2|t^{\omega^{N-1}}\!\!x,p) \cdots 
 \Phi^{N-1}(s_N|t^{\omega^1}\!\!x,p)  |0\rangle\\
 =&\prod_{1\leq i< j\leq N}
 {(\kappa^{(j-i)/N} s_j/s_i;q,\kappa)_\infty \over (\kappa^{(j-i)/N} t s_j/s_i;q,\kappa)_\infty }
{(\kappa^{(j-i)/N} q s_j/s_i;q,\kappa)_\infty \over (\kappa^{(j-i)/N} qs_j/ts_i;q,\kappa)_\infty }\cdot 
f^{\widehat{\mathfrak gl}_N} (x,p^{1/N}|s,\kappa^{1/N}|q,t). \nonumber
\end{align*}
\end{thm}

\begin{proof}
For $i=0,1,\ldots,N-1$, 
set
$
|\lambda|^{(i)}=\sum_{j\equiv i+1\,\,({\rm mod}\,N)}\lambda_j.
$ 
It follows from Lemmas \ref{lam-alpha}, \ref{lem-phi-1} and \ref{lem-phi-2} below,
we have
\begin{align*}
&{\rm LHS}=\sum_{\lambda^{(1)},\ldots,\lambda^{(N)}\in \mathsf{P}}
\langle 0| \Phi^0_{\lambda^{(N)}}(s_1) \Phi^1_{\lambda^{(N-1)}}(s_2) \cdots 
 \Phi^{N-1}_{\lambda^{(1)}}(s_N)  |0\rangle\\
 &\quad \!\!\!\times \prod_{j=1}^N \prod_{i\geq 1} (p^{1/N} x_{j+i}/tx_{j+i-1})^{\lambda^{(j)}_i}\cdot
\prod_{i=1}^N t^{|\lambda^{(i)}|^{(0)}}
 \cdot \prod_{0\leq i< j \leq N-1} t^{|\lambda^{(N-i)}|^{(N+i-j)}+|\lambda^{(N-j)}|^{(-i+j-1)}}= {\rm RHS}.
\end{align*}
\end{proof}

\begin{lem}\label{lam-alpha}
For $\alpha=0,1,\ldots,N-1$, we have
\begin{align*}
\sum_{j\geq i\geq 1\atop j-i\equiv \alpha\,\,({\rm mod}\,N)} \lambda_j-\lambda_{j+1}
=&
\sum_{j\geq 1}\left\lfloor {j+N-1-\alpha\over N}\right\rfloor \lambda_j-\left\lfloor {j+N-1-\alpha\over N}\right\rfloor \lambda_{j+1}=
|\lambda|^{(\alpha)}.\nonumber\\
\end{align*}
\end{lem}

\begin{lem}\label{lem-phi-1}
For $k=0,\ldots,N-1$, 
we have
\begin{align}
\Phi^k_{\lambda}(z)&=t^{-|\lambda|^{(0)}}
 {\Nk^{(0)}_{\lambda\lambda} (t|q,\kappa^{1/N})\over \Nk^{(0)}_{\lambda\lambda} (1|q,\kappa^{1/N})}
:\phi^k_{\lambda}(z):.\label{phi-1}
\end{align}
\end{lem}

\begin{lem}\label{lem-phi-2}
Let $0\leq \alpha<\beta\leq N-1$, we have
\begin{align}
&:\Phi^\alpha_{\lambda}(z):: \Phi^\beta_{\mu}(w):\,\,=t^{-|\lambda|^{(N+\alpha-\beta)}}t^{-|\mu|^{(-\alpha+\beta-1)}}
{(\kappa^{(\beta-\alpha)/N} w/z;q,\kappa)_\infty \over (\kappa^{(\beta-\alpha)/N} t w/z;q,\kappa)_\infty }
{(\kappa^{(\beta-\alpha)/N} q w/z;q,\kappa)_\infty \over (\kappa^{(\beta-\alpha)/N} qw/tz;q,\kappa)_\infty }\nonumber\\
&\qquad \times
{\Nk^{(\beta-\alpha)}_{\mu\lambda} (t w/z|q,\kappa^{1/N})\over \Nk^{(\beta-\alpha)}_{\mu\lambda} (w/z|q,\kappa^{1/N})}
{\Nk^{(\alpha-\beta)}_{\lambda\mu} (t z/w|q,\kappa^{1/N})\over 
\Nk^{(\alpha-\beta)}_{\lambda\mu} (z/w|q,\kappa^{1/N})}
:\Phi^\alpha_{\lambda}(z)\Phi^\beta_{\mu}(w):.\label{phi-2}
\end{align}
\end{lem}

Proofs of Lemmas \ref{lem-phi-1} and \ref{lem-phi-2} will be given in the next subsection.

\subsection{Proofs of Lemmas \ref{lem-phi-1} and \ref{lem-phi-2} }

\noindent
{\it Proof of Lemma  \ref{lem-phi-1}}.
Let $\ell=\ell(\lambda)$. 
Write $z_i=\kappa^{i/N} q^{\lambda_i}z$ for short.
We have
\begin{align*}
\Phi^k_{\lambda}(z)
=&\left( {(q/t;q)_\infty \over (q;q)_\infty}\right)^{\ell}\phi_{k-\ell}(z_{\ell+1})
\prod^{\curvearrowleft}_{1\leq i\leq \ell}S_{k-i+1}(z_i)\\
=&\,\, :\phi_{k-\ell}(z_{\ell+1})
\prod^{\curvearrowleft}_{1\leq i\leq \ell}S_{k-i+1}(z_i):\left( {(q/t;q)_\infty \over (q;q)_\infty}\right)^{\ell}\\
&\times
\prod_{1\leq i\leq \ell\atop 
 k-\ell +1\equiv k-i+1\,\,({\rm mod}\,N) } 
 {(\kappa^{1/N}qz_{i}/z_{\ell+1};q)_\infty \over (\kappa^{1/N}qz_{i}/tz_{\ell+1};q)_\infty}
 \cdot
 \prod_{1\leq i\leq \ell\atop 
k- \ell\equiv k-i+1\,\,({\rm mod}\,N) } 
 {(z_{i}/z_{\ell+1};q)_\infty \over (tz_{i}/z_{\ell+1};q)_\infty}\\
 &\times \prod_{1\leq i<j\leq \ell \atop k-j+2\equiv k-i+1\,\,({\rm mod}\,N)}
  {(\kappa^{1/N}qz_{i}/z_{j};q)_\infty \over (\kappa^{1/N}qz_{i}/tz_{j};q)_\infty}
  \cdot
   \times \prod_{1\leq i<j\leq \ell\atop k-j+1\equiv k-i+2\,\,({\rm mod}\,N)}
  {(\kappa^{-1/N}tz_{i}/z_{j};q)_\infty \over (\kappa^{-1/N}z_{i}/z_{j};q)_\infty}\\
   &\times \prod_{1\leq i<j\leq \ell \atop k-j+1\equiv k-i+1\,\,({\rm mod}\,N)}
  {(z_{i}/z_{j};q)_\infty \over (tz_{i}/z_{j};q)_\infty}
    {(qz_{i}/tz_{j};q)_\infty \over (qz_{i}/z_{j};q)_\infty}.
\end{align*}
We separate the factors in two groups, and simplify each of them as follows.
First, we have
\begin{align*}
 &
 \prod_{1\leq i\leq \ell\atop 
\ell+1-i \equiv 0\,\,({\rm mod}\,N) } 
 {(z_{i}/z_{\ell+1};q)_\infty \over (tz_{i}/z_{\ell+1};q)_\infty}\cdot
 \prod_{1\leq i<j\leq \ell\atop j-i\equiv -1\,\,({\rm mod}\,N)}
  {(\kappa^{-1/N}tz_{i}/z_{j};q)_\infty \over (\kappa^{-1/N}z_{i}/z_{j};q)_\infty}\cdot
\prod_{1\leq i<j\leq \ell \atop j-i\equiv 0\,\,({\rm mod}\, N)}
  {(z_{i}/z_{j};q)_\infty \over (tz_{i}/z_{j};q)_\infty}\\
  =&
    \prod_{1\leq i\leq j\leq \ell\atop j-i\equiv -1\,\,({\rm mod}\,N)}
  {(\kappa^{-1/N}tz_{i}/z_{j};q)_\infty \over (\kappa^{-1/N}z_{i}/z_{j};q)_\infty}
   {(z_{i}/z_{j+1};q)_\infty \over (tz_{i}/z_{j+1};q)_\infty}
       =
    \prod_{1\leq i<j\leq \ell\atop j-i\equiv -1\,\,({\rm mod}\,N)}
  {(\kappa^{(i-j-1)/N}q^{\lambda_i-\lambda_j} t;q)_{\lambda_j-\lambda_{j+1}} \over 
  (\kappa^{(i-j-1)/N}q^{\lambda_i-\lambda_j};q)_{\lambda_j-\lambda_{j+1}}}.
\end{align*}
Next, we have
\begin{align*}
&\left( {(q/t;q)_\infty \over (q;q)_\infty}\right)^{\ell}
\prod_{1\leq i\leq \ell\atop 
 \ell +1-i\equiv 1\,\,({\rm mod}\, N) } 
 {(\kappa^{1/N}qz_{i}/z_{\ell+1};q)_\infty \over (\kappa^{1/N}qz_{i}/tz_{\ell+1};q)_\infty}\\
&\times 
\prod_{1\leq i<j\leq \ell \atop j-i\equiv 1\,\,({\rm mod}\,N)}
  {(\kappa^{1/N}qz_{i}/z_{j};q)_\infty \over (\kappa^{1/N}qz_{i}/tz_{j};q)_\infty}
\cdot
 \prod_{1\leq i<j\leq \ell \atop j-i\equiv 0\,\,({\rm mod}\, N)}
    {(qz_{i}/tz_{j};q)_\infty \over (qz_{i}/z_{j};q)_\infty}\\
    =&
\prod_{1\leq i\leq j\leq \ell\atop j-i\equiv 0\,\,({\rm mod}\,N)}
 {(\kappa^{1/N}qz_{i}/z_{j+1};q)_\infty \over (\kappa^{1/N}qz_{i}/tz_{j+1};q)_\infty}
  {(qz_{i}/tz_{j};q)_\infty \over (qz_{i}/z_{j};q)_\infty}        
  =
\prod_{1\leq i\leq j\leq \ell\atop j-i\equiv 0\,\,({\rm mod}\,N)}
  {(\kappa^{(i-j)/N}q^{\lambda_i-\lambda_{j}+1}/t;q)_{\lambda_j-\lambda_{j+1}} \over 
  (\kappa^{(i-j)/N}q^{\lambda_i-\lambda_{j}+1};q)_{\lambda_j-\lambda_{j+1}} }.
 \end{align*}
 Using Lemma \ref{lam-alpha}, we have (\ref{phi-1}). \qed
 \medskip
 
 \noindent
{\it Proof of Lemma } \ref{lem-phi-2}.
It is sufficient to consider the case $\ell=\ell(\lambda)=\ell(\mu)$.
For simplicity, set $z_i=\kappa^{i/N}q^{\lambda_i}z$, and $w_i=\kappa^{i/N}q^{\mu_i}z$,
meaning $:\phi^\alpha_{\lambda}(z):\,\,=\,\, :\phi_{\alpha-\ell}(z_{\ell+1})
\prod_{1\leq i\leq \ell}S_{\alpha-i+1}(z_i):$, and 
$: \phi^\beta_{\mu}(w):\,\,=\,\, :\phi_{\beta-\ell}(w_{\ell+1})
\prod_{1\leq j\leq \ell}S_{\beta-j+1}(w_j):$. 
We have
\begin{align*}
&
{(\kappa^{(\beta-\alpha)/N} tw/z;q,\kappa)_\infty \over (\kappa^{(\beta-\alpha)/N}  w/z;q,\kappa)_\infty }
{(\kappa^{(\beta-\alpha)/N} q w/tz;q,\kappa)_\infty \over (\kappa^{(\beta-\alpha)/N} qw/z;q,\kappa)_\infty }
:\phi^\alpha_{\lambda}(z):: \phi^\beta_{\mu}(w):\\
=& \prod_{1\leq j\leq \ell \atop \alpha-\ell +1\equiv \beta-j+1\,\,({\rm mod}\,N)}
{(\kappa^{1/N}qw_j/z_{\ell+1};q)_\infty  \over (\kappa^{1/N}qw_j/tz_{\ell+1};q)_\infty } \cdot 
 \prod_{1\leq j\leq \ell \atop \alpha-\ell \equiv \beta-j+1\,\,({\rm mod}\,N)}
{(w_j/z_{\ell+1};q)_\infty  \over (tw_j/z_{\ell+1};q)_\infty }\\
&\times  \prod_{1\leq i\leq \ell \atop  \alpha-i+1\equiv \beta-\ell +1\,\,({\rm mod}\,N)}
{(\kappa^{-1/N}tw_{\ell+1}/z_{i};q)_\infty  \over (\kappa^{-1/N}w_{\ell+1}/z_{i};q)_\infty } \cdot 
 \prod_{1\leq i\leq \ell \atop  \alpha- i+1\equiv \beta-\ell \,\,({\rm mod}\,N)}
{(qw_{\ell+1}/tz_{i};q)_\infty  \over (qw_{\ell+1}/z_{i};q)_\infty }\\
&\times
\prod_{1\leq i,j\leq \ell\atop \alpha-i+2\equiv \beta-j+1\,\,({\rm mod}\,N)}
{(\kappa^{1/N} qw_j/z_i;q)_\infty \over (\kappa^{1/N} qw_j/tz_i;q)_\infty} \cdot
\prod_{1\leq i,j\leq \ell\atop \alpha-i+1\equiv \beta-j+2\,\,({\rm mod}\,N)}
{(\kappa^{-1/N} tw_j/z_i;q)_\infty \over (\kappa^{-1/N} w_j/z_i;q)_\infty} \\
&\times 
\prod_{1\leq i,j\leq \ell\atop \alpha-i+1\equiv \beta-j+1\,\,({\rm mod}\,N)}
{( w_j/z_i;q)_\infty \over ( tw_j/z_i;q)_\infty} {( qw_j/tz_i;q)_\infty \over ( qw_j/z_i;q)_\infty} 
:\phi^\alpha_{\lambda}(z)\phi^\beta_{\mu}(w):.
\end{align*}
Separate the factors in two groups, and simplify each of them as follows. 
First, we have
\begin{align*}
& 
 \prod_{1\leq j\leq \ell \atop  \alpha-\beta+1\equiv  \ell+1-j\,\,({\rm mod}\,N)}
{(w_j/z_{\ell+1};q)_\infty  \over (tw_j/z_{\ell+1};q)_\infty }\cdot
\prod_{1\leq i\leq \ell \atop \alpha-\beta-1\equiv  i-\ell-1\,\,({\rm mod}\,N)}
{(\kappa^{-1/N}tw_{\ell+1}/z_{i};q)_\infty  \over (\kappa^{-1/N}w_{\ell+1}/z_{i};q)_\infty }\\
&\times
\prod_{1\leq i,j\leq \ell\atop  \alpha-\beta-1\equiv  i-j\,\,({\rm mod}\,N)}
{(\kappa^{-1/N} tw_j/z_i;q)_\infty \over (\kappa^{-1/N} w_j/z_i;q)_\infty} 
\cdot
\prod_{1\leq i,j\leq \ell\atop \alpha-\beta\equiv  i-j\,\,({\rm mod}\,N)}
 {( w_j/z_i;q)_\infty \over ( tw_j/z_i;q)_\infty} \\
  =&
 \prod_{1\leq j\leq i\leq \ell\atop  \alpha-\beta-1\equiv  i-j\,\,({\rm mod}\,N)}
{(\kappa^{(j-i-1)/N}  q^{\mu_j-\lambda_i}tw/z;q)_{\lambda_i-\lambda_{i+1}} \over
 (\kappa^{(j-i-1)/N}q^{\mu_j-\lambda_i}w/z;q)_{\lambda_i-\lambda_{i+1}} } \\
 &\times 
  \prod_{1\leq i\leq j\leq \ell\atop  \alpha-\beta\equiv  i-j\,\,({\rm mod}\,N)}
{(\kappa^{(j-i)/N} q^{\mu_{j+1}-\lambda_{i}}tw/z;q)_{\mu_j-\mu_{j+1}} \over 
(\kappa^{(j-i)/N} q^{\mu_{j+1}-\lambda_{i}}w/z;q)_{\mu_j-\mu_{j+1}}}  
={\Nk^{(\beta-\alpha)}_{\mu\lambda}(tw/z|q,\kappa^{1/N}) \over \Nk^{(\beta-\alpha)}_{\mu\lambda}(w/z|q,\kappa^{1/N}) }.
 \end{align*}
Next, we have
\begin{align*}
&\prod_{1\leq j\leq \ell \atop \alpha-\ell +1\equiv \beta-j+1\,\,({\rm mod}\,N)}
{(\kappa^{1/N}qw_j/z_{\ell+1};q)_\infty  \over (\kappa^{1/N}qw_j/tz_{\ell+1};q)_\infty }
\cdot \prod_{1\leq i\leq \ell \atop  \alpha- i+1\equiv \beta-\ell \,\,({\rm mod}\,N)}
{(qw_{\ell+1}/tz_{i};q)_\infty  \over (qw_{\ell+1}/z_{i};q)_\infty }\\
&\times
\prod_{1\leq i,j\leq \ell\atop \alpha-i+2\equiv \beta-j+1\,\,({\rm mod}\,N)}
{(\kappa^{1/N} qw_j/z_i;q)_\infty \over (\kappa^{1/N} qw_j/tz_i;q)_\infty} 
\cdot
\prod_{1\leq i,j\leq \ell\atop \alpha-i+1\equiv \beta-j+1\,\,({\rm mod}\,N)}
{( qw_j/tz_i;q)_\infty \over ( qw_j/z_i;q)_\infty} \\
=&
\prod_{1\leq i\leq j\leq \ell\atop \alpha-\beta+1\equiv i-j\,\,({\rm mod}\,N)}
{(\kappa^{(j-i+1)/N} q^{\mu_{j+1}-\lambda_i+1}w/tz;q)_{\mu_j-\mu_{j+1}} \over 
(\kappa^{(j-i+1)/N} q^{\mu_{j+1}-\lambda_i+1}w/z;q)_{\mu_j-\mu_{j+1}} } \\
& \times 
\prod_{1\leq j\leq i\leq \ell\atop \alpha-\beta\equiv i-j\,\,({\rm mod}\,N)}
{(\kappa^{(j-i)/N} q^{\mu_{j}-\lambda_i+1}w/tz;q)_{\lambda_i-\lambda_{i+1}} \over 
(\kappa^{(j-i)/N} q^{\mu_{j}-\lambda_i+1}w/z;q)_{\lambda_i-\lambda_{i+1}}} \\
=&\,\,
t^{-|\mu|^{(-\alpha+\beta-1)}}t^{-|\lambda|^{(N+\alpha-\beta)}}
{\Nk^{(\alpha-\beta)}_{\lambda\mu}(tz/w|q,\kappa^{1/N}) \over \Nk^{(\alpha-\beta)}_{\lambda\mu}(z/w|q,\kappa^{1/N}) },
\end{align*}
where we have used Lemma \ref{lam-alpha}.
Hence we have (\ref{phi-2}). \qed

\section{Affine Laumon spaces}\label{affine-Laumon-space}

\subsection{Parabolic sheaves and affine Laumon spaces}
We briefly recall the basic facts concerning the affine Laumon spaces studied in \cite{FFNR}. 
Let ${\bf C}$ and ${\bf X}$ be smooth projective curves of genus zero.
Fix a coordinate $z$ (resp. $y$) on ${\bf C}$ (resp.  ${\bf X}$) and consider the action of $\mathbb{C}^*$ on 
${\bf C}$  (resp.  ${\bf X}$)
such that $v(z)=v^{-2} z$  (resp. $c(y)=c^{-2} y$).
We have ${\bf C}^{\mathbb{C}^*}=\{0_{\bf C},\infty_{\bf C}\}$ 
and ${\bf X}^{\mathbb{C}^*}=\{0_{\bf X},\infty_{\bf X}\}$. 
Let ${\bf S}={\bf C} \times {\bf X}$, 
${\bf D}_\infty={\bf C} \times \infty_{\bf X}\cup \infty_{\bf C}\times {\bf X}$, and
${\bf D}_0={\bf C} \times 0_{\bf X}$. Let $W$ be an $N$-dimensional vector 
space with a basis $w_1,\ldots,w_N$. 
Let $\widetilde{T}$ be the Cartan torus
acting on $W$ as follows:
for $\underline{t}=(t_1,\ldots,t_n)\in \widetilde{T}$ we have $\underline{t} (w_i)=t_i^2 w_i$.

Let $\underline{d}=(d_0,\ldots,d_{N-1})\in \mathbb{Z}_{\geq 0}^N$. 
A parabolic sheaf ${\mathcal F}_\bullet$ of degree $\underline{d}$ 
is an infinite flag of torsion free coherent sheaves of rank $N$ on 
${\bf S}:\dots\subset {\mathcal F}_{-1}\subset {\mathcal F}_{0}\subset {\mathcal F}_{1}\subset \cdots$ satisfying:
\begin{enumerate}
\item[(a)] ${\mathcal F}_{k+N}={\mathcal F}_{k}({\bf D}_0)$ for any $k$,
\item[(b)] $ch_1({\mathcal F}_{k})=k[{\bf D}_0]$ for any $k$,
\item[(c)] $ch_2({\mathcal F}_{k})=d_i$ for $i\equiv k\,\,({\rm mod}\, N)$,
\item[(d)] ${\mathcal F}_{0}$ is locally free at ${\bf D}_\infty$ and trivialized at 
${\bf D}_\infty: {\mathcal F}_{0}|_{{\bf D}_\infty}=W\otimes {\mathcal O}_{{\bf D}_\infty}$,
\item[(e)] For $-N\leq k\leq 0$ the sheaf ${\mathcal F}_{k}$ is locally free at 
${\bf D}_\infty$, 
and the quotient sheaves ${\mathcal F}_{k}/{\mathcal F}_{-N}$,
${\mathcal F}_{0}/{\mathcal F}_{k}$ (both supported at ${\bf D}_0$) are both locally free at the point 
$\infty_{\bf C}\times 0_{\bf X}$; moreover the local sections of 
${\mathcal F}_{k}|_{\infty_{\bf C}\times {\bf X}}$ are those sections of
${\mathcal F}_{0}|_{\infty_{\bf C}\times {\bf X}}=W\otimes {\mathcal O}_{\bf X}$ 
which take value in $\langle w_1,\ldots,w_{N+k} \rangle \subset W$ at $0_{\bf X} \in {\bf X}$. 
\end{enumerate}

The fine moduli space ${\mathcal P}_{\underline{d}}$ of degree $\underline{d}$ parabolic 
sheaves exists, and is a smooth connected quasiprojective variety of dimension $2d_0+\cdots 2 d_{N-1}$. 
The ${\mathcal P}_{\underline{d}}$ is called the affine Laumon space.

\subsection{Fixed points in ${\mathcal P}_{\underline{d}}$}
The group $\widetilde{T}\times \mathbb{C}^*\times\mathbb{C}^*$ acts on 
${\mathcal P}_{\underline{d}}$, with the fixed point set being finite. 
The fixed points associated with the action of $\mathbb{C}^*\times\mathbb{C}^*$ on the 
Hilbert scheme of $({\bf C}-\infty_{\bf C})\times ({\bf X}-\infty_{\bf X})$ 
are parametrized by the partitions. Namely for  
$\lambda=(\lambda_1,\lambda_2,\ldots)\in {\mathsf P}$, we have the corresponding ideal 
$J_\lambda=\mathbb{C}[z]\cdot(\mathbb{C} y^0 z^{\lambda_1}\oplus \mathbb{C}y^1 z^{\lambda_2}\oplus\cdots)$.
Write $\lambda\supset \mu$ for indicating $\lambda_i\geq \mu_i$ ($i\geq 1$), and 
write $\lambda\widetilde{\supset} \mu$ for $\lambda_i\geq \mu_{i+1}$ ($i\geq 1$).

Let ${\boldsymbol \lambda}=(\lambda^{kl})_{1\leq k,l\leq N}$ be a collection of 
partitions satisfying
\begin{align}
&\lambda^{11}\supset \lambda^{21}\supset\cdots \supset \lambda^{N1}\widetilde{\supset}\lambda^{11}, \,\,
\lambda^{22}\supset \lambda^{32}\supset\cdots \supset \lambda^{12}\widetilde{\supset}\lambda^{22}, 
\cdots ,\nonumber\\
&\qquad\qquad \lambda^{NN}\supset \lambda^{1N}\supset\cdots \supset \lambda^{N-1,N}\widetilde{\supset}\lambda^{NN}.
\label{collection}
\end{align}
Set $d_k({\boldsymbol \lambda})=\sum_{l=1}^N |\lambda^{kl}|$, and 
$\underline{d}=(d_0({\boldsymbol \lambda}),\ldots, d_{N-1}({\boldsymbol \lambda}))$,
where $d_0({\boldsymbol \lambda}):=d_N({\boldsymbol \lambda})$.

For a collection $\boldsymbol \lambda$ satisfying (\ref{collection}),  let 
${\mathcal F}_\bullet={\mathcal F}_\bullet({\boldsymbol \lambda})$ be the parabolic sheaf
\begin{align*}
{\mathcal F}_{k-N}=\bigoplus_{1\leq l\leq k} J_{\lambda^{kl}}w_l \oplus 
\bigoplus_{k\leq l\leq N} J_{\lambda^{kl}}(-{\bf D}_0)w_l. 
\end{align*}
The correspondence ${\boldsymbol \lambda} \mapsto {\mathcal F}_\bullet({\boldsymbol \lambda})$
is a bijection between the set of collections ${\boldsymbol \lambda}$ satisfying (\ref{collection}) and 
$\underline{d}({\boldsymbol \lambda})=\underline{d}$, and the set of 
$\widetilde{T}\times \mathbb{C}^*\times\mathbb{C}^*$-fixed points in ${\mathcal P}_{\underline{d}}$.

We have the bijection between the set of collection of partitions satisfying (\ref{collection})
and the set of collection of partitions $(\lambda^{(1)},\ldots,\lambda^{(N)})$, given by
$\lambda^{(l)}_{N(i-1)-N\lfloor {k-l\over N}\rfloor+k-l+1}=\lambda^{kl}_i$.
By an abuse of notation we also write ${\boldsymbol \lambda}=(\lambda^{(1)},\ldots,\lambda^{(N)})$.

\subsection{Character associated with the series $f^{\widehat{\mathfrak gl}_N}(x,p|s,\kappa|q,t)$}

For any product of the form 
$P=\prod_{a,b\geq 0} (1-q^{a}\kappa^{b} u)^{n_{ab}}$, set
$L(P)=\sum_{a,b\geq 0} n_{ab}q^{- a}\kappa^{- b} u^{-1}$.

\begin{prp}\label{ch-11}
Let ${\boldsymbol \lambda}=(\lambda^{(1)},\ldots,\lambda^{(N)})$ be an $N$-tuple of partitions.
Assume the cyclic identification as $\lambda^{(i)}=\lambda^{(i+N)}$ ($i\in \mathbb{Z}$). We have
the character ${\rm ch}({\boldsymbol \lambda})$ of the denominator in the series $f^{\widehat{\mathfrak gl}_N}(x,p|s,\kappa|q,t)$,
{\it i.e.} 
${\rm ch}({\boldsymbol \lambda}):=L\prod_{i,j=1}^N
\Nk^{(j-i|N)}_{\lambda^{(i)},\lambda^{(j)}} (s_j/s_i|q,\kappa)$ as
\begin{align*}
{\rm ch}({\boldsymbol \lambda})
=&\sum_{i=1}^N\sum_{l\leq l'\leq i}
q^{\lambda^{(l')}_{i+1-l'} -\lambda^{(l)}_{i+1-l}} \kappa^{l-l'} {s_{l}\over s_{l'}} 
q {1-q^{\lambda^{(l)}_{i+1-l} -\lambda^{(l)}_{i+2-l}} \over 1-q}\\
&
+\sum_{i=1}^N\sum_{l'<l\leq i}
q^{ \lambda^{(l')}_{i+1-l'}-\lambda^{(l)}_{i+1-l}} \kappa^{l-l'} {s_{l}\over s_{l'}} 
q {1-q^{\lambda^{(l')}_{i-l'} -\lambda^{(l')}_{i+1-l'}} \over 1-q}.\nonumber
\end{align*}
\end{prp}

Introduce a collection $(d_{ij}({\boldsymbol \lambda}))=\underline{\widetilde{d}}({\boldsymbol \lambda})$ by setting 
$d_{k,l}({\boldsymbol \lambda})=\lambda^{(l)}_{k-l+1}$. 

\begin{prp}\label{ch-22}
The ${\rm ch}({\boldsymbol \lambda})$ can be recast as
\begin{align*}
{\rm ch}({\boldsymbol \lambda})
=&
\sum_{i=1}^N
\sum_{l'\leq i-1}
\sum_{l \leq i}
 \kappa^{l-l'} {s_{l}\over s_{l'}} 
q {(1-q^{d_{i-1,l'}})(1-q^{-d_{i,l}}) \over 1-q}
+\sum_{i=1}^N
\sum_{l'\leq i-1}
 \kappa^{k-l'} {s_{k}\over s_{l'}} 
q {1-q^{d_{i-1,l'}} \over 1-q}\\
&-
\sum_{i=1}^N
\sum_{l'\leq i}
\sum_{l \leq i}
 \kappa^{l-l'} {s_{l}\over s_{l'}} 
q {(1-q^{d_{i,l'}})(1-q^{-d_{i,l}}) \over 1-q}
-
\sum_{i=1}^N
\sum_{l\leq i}
 \kappa^{l-k} {s_{l}\over s_{k}} 
q {1-q^{-d_{i,l}} \over 1-q}.\nonumber
\end{align*}
\end{prp}
Note that we have
$ \kappa^{l-l'} {s_{l}\over s_{l'}} \big|_{s\rightarrow \kappa^\delta s}=\kappa^{N \lfloor {l\over N}\rfloor-N \lfloor {l'\over N}\rfloor }
 {s_{l}\over s_{l'}} $, where $\kappa^\delta s=(\kappa^{N-1}s_1,\ldots,s_N)$.

\begin{thm}\label{affine-Laumon}
With the identification 
$t_i^2=\kappa^{N-i}s_i$, $q'=\kappa^{N}$ ($q$ being the same for both), 
the ${\rm ch}({\boldsymbol \lambda})$ 
coincides  with the torus character in a fixed tangent space to $\mathcal{P}_{\underline{d}}$.
(See Proposition {\bf 4.15} and Remark {\bf 4.17} in \cite{FFNR}.) 
Hence the Euler characteristic 
${\mathfrak J}_{\underline{d}}(s,\kappa|q,t)
:= [H^\bullet ( \mathcal{P}_{\underline{d}},\Omega_{\mathcal{P}_{\underline{d}}}^\bullet)]$
of the de Rham complex on $\mathcal{P}_{\underline{d}}$ is given 
via the Atiyah-Bott-Lefschetz localization technique as
\begin{align*}
&{\mathfrak J}_{\underline{d}}(s,\kappa|q,t)
=\sum_{i,j}(-1)^{i+j} t^j [H^i ( \mathcal{P}_{\underline{d}},\Omega_{\mathcal{P}_{\underline{d}}}^j)]
=
\sum_{{\boldsymbol \lambda}\atop \underline{d}=\underline{d}(\boldsymbol \lambda)}
\prod_{i,j=1}^N
{\Nk^{(j-i|N)}_{\lambda^{(i)},\lambda^{(j)}} (s_j/ts_i|q,\kappa) \over \Nk^{(j-i|N)}_{\lambda^{(i)},\lambda^{(j)}} (s_j/s_i|q,\kappa)},
\end{align*}
where $[H^i ( \mathcal{P}_{\underline{d}},\Omega_{\mathcal{P}_{\underline{d}}}^j)]$ denotes the 
character of $H^i ( \mathcal{P}_{\underline{d}},\Omega_{\mathcal{P}_{\underline{d}}}^j)$ as 
a representation of $\widetilde{T}\times \mathbb{C}^*\times\mathbb{C}^*$.
\end{thm}

\begin{prp}\label{affine-Laumon-2}
The non-stationary Ruijsenaars function is the generating function for the Euler characteristics of
the affine Laumon spaces
\begin{align*}
f^{\widehat{\mathfrak gl}_N}(x,p|s,\kappa|q,1/t)=
\sum_{\underline{d}} {\mathfrak J}_{\underline{d}}(s,\kappa|q,t) \prod_{i=1}^N (p t x_{i+1}/x_i)^{d_i}.
\end{align*}
\end{prp}

A geometric construction is given in \cite{BFS} for the Macdonald functions based on the Laumon spaces. 
Theorem \ref{affine-Laumon} is an affine analogue of it.

\subsection{Proofs of Theorem \ref{ch-glN} and Proposition \ref{ch-gl1}}\label{proofs-ch}
Let $K$ and  $\mu=(\mu_1,\ldots,\mu_N)$ be as in Definition \ref{dominant-integrable-weights}.
Set
$s=(\kappa t)^\delta q^\mu=q^{-K\delta/N+\mu}$ and  $\kappa=q^{-K/N} t^{-1}$. 
Then we want to show that
\begin{align*}
\lim_{t\rightarrow q} x^\mu
f^{\widehat{\mathfrak gl}_N}(x,p|q^{-K\delta/N+\mu},q^{-K/N} t^{-1}|q,q/t)={1\over (p^N;p^N)_\infty} \cdot
{\rm ch}^{\widehat{\mathfrak sl}_N}_{L(\Lambda(K,\mu))}.
\end{align*}

\noindent
{\it Proof of Theorem} \ref{ch-glN}. 
First we impose the conditions (\ref{dominant}) while $q$ and $t$ still being independent. 
We need to investigate the vanishing conditions for the numerators of the 
coefficients of the series $f^{\widehat{\mathfrak gl}_N}(x,p|q^{-K\delta/N+\mu},q^{-K/N} t^{-1}|q,q/t)$. 
We have 
\begin{align*}
&
\prod_{i,j=1}^N
\Nk^{(j-i|N)}_{\lambda^{(i)},\lambda^{(j)}} (qs_j/ts_i|q,\kappa) \\
=&\prod_{i,j=1}^N
\prod_{\beta\geq \alpha\geq 1\atop \beta-\alpha \equiv j-i \,\,({\rm mod}\,N)}
( q^{-{K\over N}(i-j-\alpha+\beta)+\mu_j-\mu_i-\lambda^{(j)}_\alpha+\lambda^{(i)}_{\beta+1}} t^{\alpha-\beta-1};q)_{\lambda^{(i)}_\beta-\lambda^{(i)}_{\beta+1}}\\
&\times\prod_{\beta\geq \alpha\geq 1\atop \beta-\alpha \equiv N-j+i-1 \,\,({\rm mod}\,N)}
( q^{-{K\over N}(i-j+\alpha-\beta-1)+\mu_j-\mu_i+\lambda^{(i)}_\alpha-\lambda^{(j)}_{\beta}} t^{\beta-\alpha};q)_{\lambda^{(j)}_\beta-\lambda^{(j)}_{\beta+1}}.
\end{align*}
Note that the terms containing $t^\alpha$ with $\alpha\neq 0$ do not contribute for vanishing. 
Then we find that the vanishing is caused only from
\begin{align*}
&\prod_{\alpha=1}^\infty
( q^{K+\mu_N-\mu_{1}+\lambda^{(1)}_\alpha-\lambda^{(N)}_{\alpha}};q)_{\lambda^{(N)}_\alpha-\lambda^{(N)}_{\alpha+1}}\cdot
\prod_{j=1}^{N-1}
\prod_{\alpha=1}^\infty
( q^{\mu_j-\mu_{j+1}+\lambda^{(j+1)}_\alpha-\lambda^{(j)}_{\alpha}} ;q)_{\lambda^{(j)}_\alpha-\lambda^{(j)}_{\alpha+1}}.
\end{align*}
Hence the condition for the vanishing is equivalent to the set of inequalities
\begin{align*}
&\lambda^{(N)}_\alpha-\lambda^{(1)}_\alpha \leq K+\mu_N-\mu_{1} \qquad (\alpha\geq 1),\\
&\lambda^{(j)}_\alpha-\lambda^{(j+1)}_\alpha \leq \mu_j-\mu_{j+1} \qquad (1\leq j<N,\alpha\geq 1).
\end{align*}

This is equivalent to the conditions for the affine Gelfand-Tsetlin patterns $D(\mu)$
\begin{align*}
\widetilde{\underline{d}}\in D(\mu) \,\,{\rm iff}\,\, d_{ij}-\widetilde{\mu}_j\leq d_{i+l,j+l}-\widetilde{\mu}_{j+l}\qquad  (j\leq i,l\geq 0),
\end{align*}
where $\widetilde{\mu}=(\widetilde{\mu}_i)_{i\in \mathbb{Z}}$ is the non increasing sequence $\widetilde{\mu}_i=\mu_{i\,({\rm mod}\, N)}+\lfloor {-i\over N}\rfloor K$.
It is proved in \cite{FFNR} that there is a weight preserving bijection 
between the set of affine Gelfand-Tsetlin patterns and the basis vectors of $L(\Lambda(K,\mu))\otimes {\mathcal F}$
($ {\mathcal F}$ stands for the space spanned by partitions). Note that their proof is
based on the $\widehat{\mathfrak{sl}}_N$-crystal of Tingley constructed on 
the set of cylindric plane partitions \cite{T} .

Next we take the limit $t\rightarrow q$. Note that all the ratios of the Nekrasov factors  in the limit $t\rightarrow q$ 
become one,  
whenever $\widetilde{\underline{d}}\in D(\mu) $. Hence we have 
\begin{align*}
&\lim_{t\rightarrow q}x^\mu
f^{\widehat{\mathfrak gl}_N}(x,p|q^{-K\delta/N+\mu},q^{-K/N} t^{-1}|q,q/t)\\
=&\,\,x^\mu
\sum_{\widetilde{\underline{d}}\in D(\mu) } \prod_{i=1}^N(p x_{i+1}/x_i)^{d_i}=
{1\over (p^N;p^N)_\infty} \cdot
{\rm ch}^{\widehat{\mathfrak sl}_N}_{L(\Lambda(K,\mu))}.
\end{align*}
\qed

Next, 
let $K=0,\mu=\emptyset$, indicating that have  $s_i=1$ ($1\leq i\leq N$) and $\kappa=t^{-1}$.
In this case, we show that 
\begin{align*}
f^{\widehat{\mathfrak gl}_N}(x,p|1,\ldots,1,t^{-1}|q,q/t)={1\over (p^N;p^N)_\infty}.
\end{align*}

\noindent
{\it Proof of Proposition} \ref{ch-gl1}. 
When $K=0,\mu=\emptyset$, we have the restriction $\lambda^{(i)}_\alpha=\lambda^{(j)}_\alpha$ ($1\leq i,j\leq N$). 
Write $\lambda=\lambda^{(i)}$ for short.
Note that we have 
$
\prod_{i=1}^N \Nk^{(j-i|N)}_{\lambda,\lambda} (u|q,1/t) = \Nk_{\lambda,\lambda} (u|q,1/t),
$
and $\Nk_{\lambda,\lambda} (1|q,1/t)=(t/q)^{|\lambda|}\Nk_{\lambda,\lambda} (q/t|q,1/t) $.
Hence $f^{\widehat{\mathfrak gl}_N}(x,p|1,\ldots,1,t^{-1}|q,q/t)$ is written as
\begin{align*}
&
\sum_{\lambda \in {\mathsf P}}
\left({ \Nk_{\lambda,\lambda} (q/t|q,1/t) \over \Nk_{\lambda,\lambda} (1|q,1/t)} \right) ^N (pq/t)^{N|\lambda|}
=\sum_{\lambda \in {\mathsf P}}  p^{N|\lambda|}=
{1\over (p^N;p^N)_\infty}.
\end{align*}
\qed

\section{Macdonald functions: the limit $p\rightarrow0$}\label{Macdonald-function}
\subsection{Macdonald functions }
We recall some facts about the Macdonald functions \cite{S,NS,BFS}.
\begin{dfn}
Let $D_x^{{\mathfrak gl}_N}=D_x^{{\mathfrak gl}_N}(q,t)$ be the Macdonald operator \cite{M} of type ${\mathfrak gl}_N$ 
\begin{align*}
D_x^{{\mathfrak gl}_N}=\sum_{i=1}^N \prod_{j\neq i} {tx_i-x_j\over x_i-x_j}T_{q,x_i},
\end{align*}
where $T_{q,x_i}$ denotes the $q$-shift operator 
$T_{q,x_i}f(x_1,\ldots,x_i,\ldots,x_N)=f(x_1,\ldots,qx_i,\ldots,x_N)$. 
\end{dfn}

Let $\mathsf{M}^{(N)}$ be the set of strictly upper triangular matrices with nonnegative integer entries:
$
\mathsf{M}^{(N)}=\{\theta=(\theta_{i,j})_{1\leq i,j\leq N}| 
\theta_{i,j}\in {\mathbb{Z}_{\geq 0}},\theta_{i,j}=0 \mbox{ if }i\geq j\}.
$
Define recursively 
$c_N(\theta;s;q,t)\in \mathbb{Q}(q,t,s_1,\cdots,s_N)$
by $c_1(-;s_1;q,t)=1$, and 
\begin{align*}
c_N(\theta\in \mathsf{M}^{(N)};s_1,\cdots,s_N;q,t)
=&\,\,c_{N-1}(\theta\in \mathsf{M}^{(N-1)};
q^{-\theta_{1,N}}s_1,\cdots,q^{-\theta_{N-1,N}}s_{N-1};q,t)\\
&\times
\prod_{1\leq i\leq j\leq N-1}
{(t s_{j+1}/s_i;q)_{\theta_{i,N}}\over (q s_{j+1}/s_i;q)_{\theta_{i,N}}}
{(q^{-\theta_{j,N}}q s_j/ts_i;q)_{\theta_{i,N}}\over (q^{-\theta_{j,N}} s_j/s_i;q)_{\theta_{i,N}}}.
\nonumber
\end{align*}
We have 
\begin{align*}
&c_N(\theta;s_1,\cdots,s_N;q,t)\\
=&
\prod_{k=2}^{N}
\prod_{1\leq i\leq j\leq k-1}
{(q^{\sum_{a=k+1}^N(\theta_{i,a}-\theta_{j+1,a})}t s_{j+1}/s_i;q)_{\theta_{i,k}}\over
(q^{\sum_{a=k+1}^N(\theta_{i,a}-\theta_{j+1,a})}qs_{j+1}/s_i;q)_{\theta_{i,k}}}
{(q^{-\theta_{j,k}+\sum_{a=k+1}^N(\theta_{i,a}-\theta_{j,a})}qs_j/ts_i;q)_{\theta_{i,k}}\over
(q^{-\theta_{j,k}+\sum_{a=k+1}^N(\theta_{i,a}-\theta_{j,a})}s_j/s_i;q)_{\theta_{i,k}}}.
\nonumber
\end{align*}

\begin{dfn}
Define $f^{{\mathfrak gl}_N}(x|s|q,t)\in \mathbb{Q}(s,q,t)[[x_2/x_1,\ldots ,x_N/x_{N-1}]]$ by
\begin{align*}
&f^{{\mathfrak gl}_N}(x|s|q,t)
= \sum_{\theta\in\mathsf{M}^{(N)}}c_N(\theta;s;q,t)
\prod_{1\leq i<j \leq N} (x_j/x_i)^{\theta_{i,j}}.
\end{align*}
\end{dfn}

\begin{prp}[\cite{NS,BFS}]
Let $\lambda=(\lambda_1,\ldots,\lambda_N)\in \mathbb{C}^N$, and set $s=t^{\delta} q^{\lambda}$ ($s_i=t^{N-i}q^{\lambda_i}$).
Then we have
\begin{align*}
D_x^{{\mathfrak gl}_N} x^{\lambda} f^{{\mathfrak gl}_N}(x|s|q,t)=
\sum_{i=1}^N s_i\,\,x^{\lambda} f^{{\mathfrak gl}_N}(x|s|q,t).
\end{align*}
\end{prp}

\begin{lem}
We have
\begin{align*}
\lim_{\epsilon \rightarrow 0} f^{{\mathfrak gl}_N}(x|\epsilon^{-\delta}s|q,t)=
\prod_{1\leq i<j\leq N}{(q x_j/ x_i;q)_\infty \over (q x_j/ tx_i;q)_\infty}.
\end{align*}
\end{lem}

\begin{proof}
In the limit $\epsilon \rightarrow 0$, we have $\epsilon^{j-i}s_j/s_i \rightarrow 0$ for $1\leq i<j\leq N$. Hence 
we have 
\begin{align*}
{\rm LHS}=\sum_{\theta\in\mathsf{M}^{(N)}} 
\prod_{1\leq i<j\leq N}
{(t;q)_{\theta_{i,j}}\over (q;q)_{\theta_{i,j}}}(q x_j/t x_i)^{\theta_{i,j}}={\rm RHS}.
\end{align*}
\end{proof}
\begin{rmk}\label{kappa=0-Mac}
We remark that Proposition \ref{kappa=0} is obtained in the same way as above.
\end{rmk}

\begin{dfn}
Let $\varphi^{{\mathfrak gl}_N}(x|s|q,t)\in \mathbb{Q}(q,t)[[x_2/x_1,\ldots ,x_N/x_{N-1},s_2/s_1,\ldots ,s_N/s_{N-1}]]$ be
\begin{align*}
&\varphi^{{\mathfrak gl}_N}(x|s|q,t)
=\prod_{1\leq i<j\leq N}{(q x_j/ tx_i;q)_\infty \over (q x_j/ x_i;q)_\infty} f^{{\mathfrak gl}_N}(x|s;q,t),
\end{align*}
where $c_N(\theta;s;q,t)$'s are 
expanded in $ \mathbb{Q}(q,t)[[s_2/s_1,\ldots ,s_N/s_{N-1}]]$.
\end{dfn}

\begin{prp}[\cite{NS}] \label{Macdonald-duality}
We have 
\begin{align*}
&\varphi^{{\mathfrak gl}_N}(x|s|q,t)=\varphi^{{\mathfrak gl}_N}(s|x|q,t)\qquad
\,\,\,\,\,\, (\mbox{bispectral duality}), \\
&\varphi^{{\mathfrak gl}_N}(x|s|q,t)=
\varphi^{{\mathfrak gl}_N}(x|s|q,q/t)\qquad (\mbox{Poincar\'e duality}).
\end{align*}
\end{prp}

\subsection{Macdonald limit of $ f^{\widehat{\mathfrak gl}_N}(x,p|s,\kappa|q,t)$: $p\rightarrow 0$}

\begin{prp}
We have
$\lim_{p\rightarrow 0} f^{\widehat{\mathfrak gl}_N}(p^\delta x,p|\kappa^\delta s,\kappa|q,t)=
f^{{\mathfrak gl}_N}(x|s|q,q/t)$.
\end{prp}

\begin{rmk}
Note that we have defined  $f^{\widehat{\mathfrak gl}_N}(x,p|s,\kappa|q,t)$ in such a way that
we obtain $f^{{\mathfrak gl}_N}(x|s|q,q/t)$ in the limit $p\rightarrow 0$, 
instead of $f^{{\mathfrak gl}_N}(x|s|q,t)$. Some explanations about this inconvenient definition is in order. 
(1) We started our construction based on the  affine screening operators with the Heisenberg algebra given in Definition \ref{bosons}.
This notation seems rather standard
in the context of the deformed $W$-algebras, 
and it seems safe and reasonable to stick to this convention. 
One may also expect that there are some 
representation theoretical interpretations for 
having the exchange of the parameters $t\leftrightarrow q/t$.
(2) We regard $f^{\widehat{{\mathfrak gl}}_N}(x,p|s,\kappa|q,t)$ as the generating function for the Euler characteristics of the affine Laumon space 
(see Proposition \ref{affine-Laumon-2}), where the parameter $1/t$ counts the degrees of the differential forms. Hence in this geometric context, the use of 
the parameter $t$ (not $q/t$) seems more natural.
\end{rmk}

\begin{proof}
While taking the  limit $p\rightarrow 0 $ of $f^{\widehat{\mathfrak gl}_N}(p^\delta x,p|\kappa^\delta s,\kappa|q,t)$, 
the partitions producing non vanishing contribution to the summation satisfy 
$\ell (\lambda^{(i)})\leq N-i$ ($1\leq i\leq N$).
Hence we can parametrize them by 
using the set $\mathsf{M}^{(N)}$ as 
$
\lambda^{(i)}_j=\sum_{k=i+j}^N \theta_{i,k}.
$
Namely $\theta_{i,j}=\lambda^{(i)}_{j-i}-\lambda^{(i)}_{j-i+1}$.
Assuming the restriction condition $\ell (\lambda^{(i)})\leq N-i$, we have for $1\leq i\leq j\leq N$
\begin{align*}
\Nk^{(j-i|N)}_{\lambda^{(i)},\lambda^{(j)}} (u|q,\kappa)=
&\prod_{1\leq l\leq k \leq N-i\atop k=l+j-i} 
(u q^{-\lambda^{(j)}_l +\lambda^{(i)}_{k+1}}  \kappa^{-l+k};q)_{\lambda^{(i)}_k-\lambda^{(i)}_{k+1}}\\
=&
\prod_{k=j+1}^N 
(u q^{-\sum_{\alpha=k}^N \theta_{j,\alpha}+ \sum_{\alpha=k+1}^N \theta_{i,\alpha}}  \kappa^{j-i};q)_{\theta_{i,k}},
\end{align*}
and for $1\leq j<i\leq N$
\begin{align*}
\Nk^{(j-i|N)}_{\lambda^{(i)},\lambda^{(j)}} (u|q,\kappa)=
&\prod_{1\leq l\leq k \leq N-j\atop k=l+i-j-1} 
(u q^{\lambda^{(i)}_l -\lambda^{(j)}_{k}}  \kappa^{l-k-1};q)_{\lambda^{(j)}_k-\lambda^{(j)}_{k+1}}\\
=&
\prod_{k=i}^N 
(u q^{\sum_{\alpha=k+1}^N \theta_{i,\alpha}- \sum_{\alpha=k}^N \theta_{j,\alpha}}  \kappa^{j-i};q)_{\theta_{j,k}}. 
\end{align*}
Note that we have the rules for changing the order of products
$\prod_{1\leq i\leq j\leq N}\prod_{k=j+1}^N=\prod_{k=2}^N \prod_{1\leq i\leq j\leq  k-1}$
and 
$\prod_{1\leq j< i\leq N}\prod_{k=i}^N=\prod_{k=2}^N \prod_{1\leq j< i\leq  k}$.
Thus we obtain 
\begin{align*}
\prod_{i,j=1}^N
{\Nk^{(j-i|N)}_{\lambda^{(i)},\lambda^{(j)}} (\kappa^{i-j}ts_j/s_i|q,\kappa) \over 
\Nk^{(j-i|N)}_{\lambda^{(i)},\lambda^{(j)}} (\kappa^{i-j}s_j/s_i|q,\kappa)}
\cdot \prod_{\beta=1}^N\prod_{\alpha= 1}^{N-\beta} (  x_{\alpha+\beta}/tx_{\alpha+\beta-1})^{\lambda^{(\beta)}_\alpha}
=c_N(\theta;s;q,q/t)\cdot \prod_{1\leq i<j \leq N} (x_j/x_i)^{\theta_{i,j}}.
\end{align*}

\end{proof}

\section{Ruijsenaars operator and its particular limits}
In this section, starting from the Ruijsenaars operator  $D_x(p)$, we take several particular limits, deriving 
the Macdonald operator, the elliptic Calogero-Sutherland operator, and the $q$-difference affine Toda  operator.

\subsection{Ruijsenaars operator}
Let  $D_x(p)=D_x(p|q,t)$ be the Ruijsenaars operator in (\ref{D-Ruijsenaars}).

\begin{dfn}
Changing the base as $p\rightarrow p^N$, and 
making the shift in $x$ as $x\rightarrow p^{-\delta}x$, set 
$\mathsf{D}_x(p)=\mathsf{D}_x(p|q,t)$ by 
\begin{align*}
\mathsf{D}_x(p)=D_{p^{-\delta} x}(p^N)=&\sum_{i=1}^N t^{N-i}
\prod_{j=1}^{i-1} 
{\Theta_{p^N}(p^{i-j} t x_i/x_j) \over \Theta_{p^N}(p^{i-j} x_i/x_j) }\cdot
\prod_{k=i+1}^N 
{\Theta_{p^N}(p^{k-i}  x_k/tx_i) \over \Theta_{p^N}(p^{k-i} x_k/x_i) } \cdot T_{q,x_i}.
\end{align*}
\end{dfn}
Note that
in terms of the modified Ruijsenaars operator $\mathsf{D}_x(p)$, 
the main conjecture (Conjecture \ref{MAIN}) is recast as
\begin{align*}
&\mathsf{D}_x(p)\, x^\lambda f^{{\rm st.}\widehat{\mathfrak gl}_N}(x,p|s|q,q/t)=
\varepsilon(p^N|s|q,t)\,x^\lambda f^{{\rm st.}\widehat{\mathfrak gl}_N}(x,p|s|q,q/t),
\end{align*}
where $s=t^\delta q^\lambda$ ($s_i=t^{N-i} q^{\lambda_i}$).

\begin{dfn}
For simplifying our calculation in the limit $q\rightarrow 1$,
set $\widetilde{\mathsf{D}}_x(p)=\widetilde{\mathsf{D}}_x(p|q,t)$ by 
\begin{align*}
\widetilde{\mathsf{D}}_x(p)=x^{\beta \delta}\mathsf{D}_x(p|q,t)x^{-\beta \delta}=&\sum_{i=1}^N
\prod_{j=1}^{i-1} 
{\Theta_{p^N}(p^{i-j} t x_i/x_j) \over \Theta_{p^N}(p^{i-j} x_i/x_j) }\cdot
\prod_{k=i+1}^N 
{\Theta_{p^N}(p^{k-i}  x_k/tx_i) \over \Theta_{p^N}(p^{k-i} x_k/x_i) } \cdot T_{q,x_i}.
\end{align*}
\end{dfn}

\subsection{Macdonald operator: limit $p\rightarrow 0$}
We have the Macdonald operator $D_x^{\mathfrak{gl}_N}(q,t)$ in the limit $p\rightarrow 0$
\begin{align*}
D_x^{\mathfrak{gl}_N}(q,t)=\lim_{p\rightarrow 0}D_x(p)=&\sum_{i=1}^N
\prod_{j\neq i}
{t x_i-x_j \over x_i-x_j } T_{q,x_i}.
\end{align*}

\subsection{Elliptic Calogero-Sutherland operator: limit $q\rightarrow 1$}
We set
$q=e^h,t=e^{\beta h}$ and consider the limit $h\rightarrow 0$ of the Ruijsenaars operator 
$\widetilde{\mathsf{D}}_x(p)$, while fixing $\beta$. 
As for the details, we refer the readers to the work of Langmann \cite{L} where the kernel function identity for 
the non-stationary Calogero-Sutherland model was introduced.

\subsubsection{Derivatives of theta function and elliptic potential $V(z|p)$}
For studying the $h$ expansion of $\widetilde{\mathsf{D}}_x(p)$, 
we collect some simple facts  concerning the derivatives of the theta function $\Theta_p(z)$.

\begin{dfn}
Set
$
\Theta^{(1)}_p(z)=z {\partial \over \partial z}\Theta_p(z)$, and 
$
\Theta^{(2)}_p(z)=\left(z {\partial \over \partial z}\right)^2\Theta_p(z).
$
\end{dfn}

\begin{lem}
We have
\begin{align*}
&\Theta_p(1/z)=\Theta_p(p z),\quad 
\Theta^{(1)}_p(1/z)=-\Theta^{(1)}_p(p z),\quad
\Theta^{(2)}_p(1/z)=\Theta^{(2)}_p(p z),\\
&\Theta_p(p z)=-z^{-1} \Theta_p(z),\quad
\Theta^{(1)}_p(p z)=z^{-1} \Theta_p(z)-z^{-1} \Theta^{(1)}_p( z),\\
&\Theta^{(2)}_p(p z)=-z^{-1} \Theta_p(z)+2 z^{-1} \Theta^{(1)}_p( z)-z^{-1} \Theta^{(2)}_p( z),\quad
\Theta^{(1)}_p(1)=
\Theta^{(2)}_p(1)=- (p;p)_\infty^3,\\
&
-p{\partial\over \partial p} \Theta_{p^2} (p z)+\left(z{\partial\over \partial z}  \right)^2 \Theta_{p^2} (p z)=0.
\end{align*}
\end{lem}

\begin{dfn}
Set
$
 V(z|p)= 
 {\Theta^{(2)}_{p}(z) \over\Theta_{p}(z)}-
\left( {\Theta^{(1)}_{p}(z) \over\Theta_{p}(z)} \right)^2.
$
\end{dfn}

\begin{prp}
The potential $V(z|p)$ is an elliptic function satisfying the conditions $V(pz|p)=V(z|p)$,
\begin{align*}
V(z|p)=&-{1\over (1-z)^2}+{1\over 1-z}+V_0(p)
+O(1-z),\\
V_0(p)=&
2 {1\over (p;p)_\infty }p {\partial (p;p)_\infty \over \partial p}
=
-
2 p-6 p^2-8 p^3-14 p^4-12 p^5-24 p^6-16 p^7-\cdots.\nonumber
\end{align*}
\end{prp}

\begin{rmk}
Note that $V(z|p)$ is essentially the Weierstrass' $\wp(u)$
\begin{align*}
V(e^{iu}|p)={1\over u^2}+{1\over 12}+V_0(p)+O(u)=\wp(u)+{1\over 12}+V_0(p).
\end{align*}
\end{rmk}

\begin{lem}
We have
\begin{align*}
V_0(p^N)={2\over 3N} {1\over \Theta_{p^N}^{(1)}(1)}p {\partial \over \partial p}\Theta_{p^N}^{(1)}(1).
\end{align*} 

\end{lem}

\subsubsection{Elliptic Sutherland Hamiltonian $H_\beta(p)$}

\begin{dfn}
Let $\vartheta_i=x_i\partial/\partial x_i$.
Set $H_\beta(p)$ as
\begin{align}
&H_\beta(p)={1\over 2} \sum_{i=1}^N \vartheta_i^2 
-\beta \sum_{1\leq i<j\leq N}  {\Theta^{(1)}_{p^N}(p^{j-i} x_j/x_i) \over\Theta_{p^N}(p^{j-i} x_j/x_i)}(\vartheta_i-\vartheta_j)
+
\beta^2 \sum_{1\leq i<j\leq N}  {\Theta^{(2)}_{p^N}(p^{j-i} x_j/x_i) \over\Theta_{p^N}(p^{j-i} x_j/x_i)}\nonumber \\
&\quad+\beta^2\sum_{1\leq i<j<k\leq N}
\left(
 {\Theta^{(1)}_{p^N}(p^{j-i} x_j/x_i) \over\Theta_{p^N}(p^{j-i} x_j/x_i)}
 {\Theta^{(1)}_{p^N}(p^{k-i} x_k/x_i) \over\Theta_{p^N}(p^{k-i} x_k/x_i)}
-
 {\Theta^{(1)}_{p^N}(p^{j-i} x_j/x_i) \over\Theta_{p^N}(p^{j-i} x_j/x_i)}
 {\Theta^{(1)}_{p^N}(p^{k-j} x_k/x_j) \over\Theta_{p^N}(p^{k-j} x_k/x_j)} \right.\nonumber\\
  &\left.
\qquad\qquad \qquad+
 {\Theta^{(1)}_{p^N}(p^{k-i} x_k/x_i) \over\Theta_{p^N}(p^{k-i} x_k/x_i)}
 {\Theta^{(1)}_{p^N}(p^{k-j} x_k/x_i) \over\Theta_{p^N}(p^{k-j} x_k/x_j)}
\right) \nonumber
.
\end{align}
\end{dfn}

\begin{prp}
Set $q=e^h,t=e^{\beta h}$. We have 
\begin{align*}
&\widetilde{\mathsf{D}}_x(p)=N+h\sum_{i=1}^N \vartheta_i+h^2 H_\beta(p)+O(h^3).
\end{align*}
\end{prp}

\begin{lem} \label{theta-identity}
We have the identity
\begin{align*}
&
 {\Theta^{(1)}_{p^N}(x_j/x_i) \over\Theta_{p^N}(x_j/x_i)}
 {\Theta^{(1)}_{p^N}(x_k/x_i) \over\Theta_{p^N}(x_k/x_i)}
-
 {\Theta^{(1)}_{p^N}(x_j/x_i) \over\Theta_{p^N}(x_j/x_i)}
 {\Theta^{(1)}_{p^N}(x_k/x_j) \over\Theta_{p^N}(x_k/x_j)}
  +
 {\Theta^{(1)}_{p^N}(x_k/x_i) \over\Theta_{p^N}(x_k/x_i)}
 {\Theta^{(1)}_{p^N}(x_k/x_j) \over\Theta_{p^N}(x_k/x_j)} \nonumber\\
 =&
 {1\over 2} \left( {\Theta^{(2)}_{p^N}(x_j/x_i) \over\Theta_{p^N}(x_j/x_i)}+
  {\Theta^{(2)}_{p^N}(x_k/x_i) \over\Theta_{p^N}(x_k/x_i)}+
 {\Theta^{(2)}_{p^N}(x_k/x_j) \over\Theta_{p^N}(x_k/x_j)}\right) \\
 &-{1\over 2}
 \left( {\Theta^{(1)}_{p^N}(x_j/x_i) \over\Theta_{p^N}(x_j/x_i)}-
  {\Theta^{(1)}_{p^N}(x_k/x_i) \over\Theta_{p^N}(x_k/x_i)}+
 {\Theta^{(1)}_{p^N}(x_k/x_j) \over\Theta_{p^N}(x_k/x_j)}\right) 
 -{1\over N} {1\over \Theta^{(1)}_{p^N} (1)} p {\partial \over \partial p}  \Theta^{(1)}_{p^N} (1).\nonumber
\end{align*}
\end{lem} 

\begin{lem}
Using the identity in Lemma \ref{theta-identity} above,
we can recast the three body interaction terms in $H_\beta$ as
\begin{align*}
&\sum_{1\leq i<j<k\leq N}\left(
 {\Theta^{(1)}_{p^N}(p^{j-i} x_j/x_i) \over\Theta_{p^N}(p^{j-i} x_j/x_i)}
 {\Theta^{(1)}_{p^N}(p^{k-i} x_k/x_i) \over\Theta_{p^N}(p^{k-i} x_k/x_i)}
-
 {\Theta^{(1)}_{p^N}(p^{j-i} x_j/x_i) \over\Theta_{p^N}(p^{j-i} x_j/x_i)}
 {\Theta^{(1)}_{p^N}(p^{k-j} x_k/x_j) \over\Theta_{p^N}(p^{k-j} x_k/x_j)} \right.\nonumber\\
&\qquad \qquad + \left. 
 {\Theta^{(1)}_{p^N}(p^{k-i} x_k/x_i) \over\Theta_{p^N}(p^{k-i} x_k/x_i)}
 {\Theta^{(1)}_{p^N}(p^{k-j} x_k/x_j) \over\Theta_{p^N}(p^{k-j} x_k/x_j)} \right)\nonumber\\
 =&
\sum_{1\leq i<j\leq N} \left(  {N-2\over 2}  {\Theta^{(2)}_{p^N}(p^{j-i} x_j/x_i) \over\Theta_{p^N}(p^{j-i} x_j/x_i)}
 - {N-2j+2i\over 2}
 {\Theta^{(1)}_{p^N}(p^{j-i} x_j/x_i) \over\Theta_{p^N}(p^{j-i} x_j/x_i)} \right)\\
 &-{(N-1)(N-2)\over 6} {1\over \Theta^{(1)}_{p^N} (1)} p {\partial \over \partial p}  \Theta^{(1)}_{p^N} (1).\nonumber
\end{align*}
\end{lem}

\begin{prp}
We have $H_\beta(p)$ written only with two body interaction terms as
\begin{align*}
H_\beta(p)=&{1\over 2} \sum_{i=1}^N \vartheta_i^2 
-\beta \sum_{1\leq i<j\leq N}  {\Theta^{(1)}_{p^N}(p^{j-i} x_j/x_i) \over\Theta_{p^N}(p^{j-i} x_j/x_i)}(\vartheta_i-\vartheta_j)\nonumber\\
&+
\beta^2 \sum_{1\leq i<j\leq N} \left(
{N\over 2}
 {\Theta^{(2)}_{p^N}(p^{j-i} x_j/x_i) \over\Theta_{p^N}(p^{j-i} x_j/x_i)}-
 {N-2j+2i\over 2}
 {\Theta^{(1)}_{p^N}(p^{j-i} x_j/x_i) \over\Theta_{p^N}(p^{j-i} x_j/x_i)} \right)\nonumber \\
&-\beta^2{(N-1)(N-2)\over 6}  {1\over \Theta^{(1)}_{p^N} (1)} p {\partial \over \partial p}  \Theta^{(1)}_{p^N} (1).
\end{align*}

\end{prp}

\subsubsection{Quasi-ground state $\psi_0(x,p)$ and elliptic Calogero-Sutherland Hamiltonian $H^{\rm eCS}(p)$ }

\begin{dfn}\label{quasi-ground-state}
Set $q=e^h,t=e^{\beta h}$.
In view of Definition \ref{normalized-form},
we introduce the quasi-ground state $\psi_0(x,p|\beta)$ as follows.
\begin{align*}
\psi_0(x,p|\beta)=&
\lim_{h\rightarrow0}
\prod_{1\leq i<j\leq N}
{(p^{j-i}qx_j/tx_i;q,p^N)_\infty \over (p^{j-i}qx_j/x_i;q,p^N)_\infty }
\cdot 
\prod_{1\leq i\leq j\leq N}
{(p^{N-j+i}qx_i/tx_j;q,p^N)_\infty \over (p^{N-j+i}qx_i/x_j;q,p^N)_\infty } \nonumber\\
=&
\left(  
\left((p^N;p^N)_\infty\right)^{N-N(N-1)/2}
\prod_{1\leq i<j\leq N}\Theta_{p^N}(p^{j-i}x_j/x_i) \right)^\beta.
\end{align*}
\end{dfn}

\begin{dfn}\label{H-eCS}
Let $H^{\rm eCS}(p)=H^{\rm eCS}(p|\beta)$ be the elliptic Calogero-Sutherland operator 
\begin{align*}
H^{\rm eCS}(p)=&
{1\over 2} \sum_{i=1}^N\vartheta_i^2
+\beta(\beta-1)  \sum_{1\leq i<j\leq N}V(p^{j-i} x_j/x_i|p^N)
+{ \beta(\beta-1)N \over 2}V_0(p^N).
\end{align*}
\end{dfn}

\begin{prp}
We have 
\begin{align*}
& \psi_0(x,p) \,\left(H_\beta(p)+{ \beta(\beta-1)N \over 2}V_0(p^N)\right) \,\, \psi_0^{-1}(x,p) 
=H^{\rm eCS}(p).
\end{align*}
\end{prp}

\subsection{Affine $q$-Toda operator: $t\rightarrow 0$}
\begin{dfn}\label{q-Toda-D}
Define the $\widehat{\mathfrak gl}_N$ affine $q$-Toda operator  
$D_x^{\widehat{\mathfrak gl}_N{\rm Toda}}=D_x^{\widehat{\mathfrak gl}_N{\rm Toda}}(q,\widetilde{p})$  by
\begin{align*}
D_x^{\widehat{\mathfrak gl}_N{\rm Toda}}&=
(1-\widetilde{p}x_2/x_1)T_{q,x_1}+(1-\widetilde{p}x_3/x_2)T_{q,x_2}+\cdots \\
&\qquad \cdots+ (1-\widetilde{p}x_N/x_{N-1})T_{q,x_{N-1}}+ (1-\widetilde{p}x_1/x_{N})T_{q,x_N}.\nonumber
\end{align*}
\end{dfn}

\begin{prp}
Putting $p=\tilde{p} t$ in $\widetilde{\mathsf{D}}_x(p)$, we have the 
affine $q$-difference Toda operator $D^{\widehat{\mathfrak{gl}}_N{\rm Toda}} (\widetilde{p})$ in
the limit $t\rightarrow 0$ with $\widetilde{p}$ being fixed
\begin{align*}
\lim_{t\rightarrow 0\atop \tilde{p} \,\,{\rm fixed}}\widetilde{\mathsf{D}}_x(\widetilde{p}t)=
D_x^{\widehat{\mathfrak{gl}}_N{\rm Toda}}(q,\widetilde{p}).
\end{align*}
\end{prp}

\section{Non-stationary and stationary $\widehat{\mathfrak gl}_N$ $q$-difference affine Toda equation: the limit $p,t\rightarrow 0$}

\subsection{$\widehat{\mathfrak gl}_N$ non-stationary affine $q$-Toda operator 
${\mathcal T}^{\widehat{\mathfrak gl}_N}(\kappa)$}
We study the limit $p,t\rightarrow 0$ while the ratio $\widetilde{p}=p/t$ being fixed.
Let $D_x^{\widehat{\mathfrak gl}_N{\rm Toda}}$ be as in Definition \ref{q-Toda-D}. 

\begin{dfn}\label{NST-op}
Denote the Euler operators by $\vartheta_i=x_i \partial/\partial x_i$ ($1\leq i\leq N$).
Let $\Delta$ be the Laplacian
$
\Delta= {1\over 2} \sum_{i=1}^N \vartheta_i^2.
$
Introduce the operator 
${\mathcal T}^{\widehat{\mathfrak gl}_N}(\kappa)={\mathcal T}^{\widehat{\mathfrak gl}_N}(q,\widetilde{p},\kappa)$ defined by
\begin{align*}
{\mathcal T}^{\widehat{\mathfrak gl}_N}(\kappa)=\prod_{i=1}^{N}{1\over  (\widetilde{p}q x_{i+1}/x_i;q)_\infty }  \cdot 
q^{\Delta}T_{\kappa,\widetilde{p}},
\end{align*}
where $T_{\kappa,\widetilde{p}}f(\widetilde{p})=f(\kappa \widetilde{p})$ or 
equivalently $T_{\kappa,\widetilde{p}}=\kappa^{ \widetilde{p} \partial/\partial  \widetilde{p}}$.
We call ${\mathcal T}^{\widehat{\mathfrak gl}_N}(\kappa)$ the non-stationary affine $q$-Toda operator.
\end{dfn}

\begin{prp}
If $\kappa=1$, then we have the commutativity
$
[{\mathcal T}^{\widehat{\mathfrak gl}_N}(1),D^{\widehat{\mathfrak gl}_N{\rm Toda}}]=0. 
$
On the other hand, 
we have $[{\mathcal T}^{\widehat{\mathfrak gl}_N}(\kappa),D^{\widehat{\mathfrak gl}_N{\rm Toda}}]\neq 0$ 
when $\kappa\neq 1$.
\end{prp}

\begin{rmk}
It seems a difficult problem to find such a $\kappa$-deformation of 
$D^{\widehat{\mathfrak gl}_N{\rm Toda}}$ that which commutes with ${\mathcal T}^{\widehat{\mathfrak gl}_N}(\kappa)$.
\end{rmk}

We call 
$D^{\widehat{\mathfrak gl}_N{\rm Toda}} \psi(x,\widetilde{p}|s|q)= 
\varepsilon(\widetilde{p}|s|q)\psi(x,\widetilde{p}|s|q)$ the 
stationary problem for the affine $q$-Toda system. 
Note that the eigenvalue $\varepsilon(\widetilde{p}|s|q)$
depends on the variable $\widetilde{p}$.
We call the eigenvalue equation 
${\mathcal T}^{\widehat{\mathfrak gl}_N}(\kappa) \psi(x,\widetilde{p}|s,\kappa|q)= \varepsilon(s|q)
 \psi(x,\widetilde{p}|s,\kappa|q)$, the 
non-stationary problem for the affine $q$-Toda system. 
One may regard this problem as a $q$-difference analogue of the Heat equation (or a time dependent Schr\"odinger equation) in an  affine Toda potential.
Note that the eigenvalue $\varepsilon(s|q)$ in the non-stationary case does not depend in $\widetilde{p}$.

\subsection{Conjecture concerning non-stationary $\widehat{\mathfrak gl}_N$ $q$-difference affine Toda equation}
\begin{dfn}
Set
\begin{align*}
f^{\widehat{\mathfrak gl}_N{\rm Toda}}(x,\widetilde{p}|s,\kappa|q)=
\lim_{t\rightarrow 0} f^{\widehat{\mathfrak gl}_N}(x,t \widetilde{p}|s,\kappa|q,q/t).
\end{align*}
\end{dfn}

\begin{con}
Let $s_i=q^{\lambda_i}$. 
We have
\begin{align}
{\mathcal T}^{\widehat{\mathfrak gl}_N}(\kappa) x^\lambda f^{\widehat{\mathfrak gl}_N{\rm Toda}}(x,\widetilde{p}|s,\kappa|q)= 
q^{{1\over 2} \sum_{i=1}^N \lambda_i^2} x^\lambda f^{\widehat{\mathfrak gl}_N{\rm Toda}}(x,\widetilde{p}|s,\kappa|q). 
\label{non-stationary-Toda-eq}
\end{align}
\end{con}

\begin{prp} \label{non-stationary-Toda}
Let $s_i=q^{\lambda_i}$. 
The Poincar\'e duality conjecture (\ref{conjecture-Poincare}) in Conjecture \ref{main-con} implies
the non-stationary eigenvalue equation (\ref{non-stationary-Toda-eq}).
\end{prp}

Recall that we have set 
$
m_i=m_i({\boldsymbol \lambda})=
\sum_{\beta=1}^N{ \sum_{\alpha\geq 1 \atop \alpha+\beta\equiv i\,\,({\rm mod}\,N)}}\lambda^{(\beta)}_\alpha-\lambda^{(\beta+1)}_\alpha.
$
Hence we have
$ \prod_{\beta=1}^N\prod_{\alpha\geq 1} (  x_{\alpha+\beta}/x_{\alpha+\beta-1})^{\lambda^{(\beta)}_\alpha}=
 \prod_{i=1}^N x_i^{m_i}.
$

\begin{prp}\label{non-stationary-Toda-2}
We have
\begin{align}
&f^{\widehat{\mathfrak gl}_N{\rm Toda}}(x,\widetilde{p}|s,\kappa|q)\label{non-stationary-Toda-f-ex}\\
=&\,\,
\sum_{\lambda^{(1)},\ldots,\lambda^{(N)}\in {\mathsf P}}
 \prod_{i=1}^N s_i^{-m_i} q^{-m_i^2/2} \kappa^{-|\lambda^{(i)}|} \cdot
\prod_{i,j=1}^N
{1 \over 
\Nk^{(j-i|N)}_{\lambda^{(i)},\lambda^{(j)}} (s_j/s_i|q,\kappa)}
\cdot \prod_{\beta=1}^N\prod_{\alpha= 1}^{N-\beta} ( \widetilde{p} x_{\alpha+\beta}/x_{\alpha+\beta-1})^{\lambda^{(\beta)}_\alpha}.
\nonumber
\end{align}
\end{prp}

\subsection{Proof of Propositions \ref{non-stationary-Toda} and  \ref{non-stationary-Toda-2}}

The Poincar\'e duality conjecture (\ref{conjecture-Poincare}) in Conjecture \ref{main-con} is recast as 
\begin{align*}
&\prod_{1\leq i<j\leq N}
{(p^{j-i}tx_j/x_i;q,p^N)_\infty \over (p^{j-i}qx_j/tx_i;q,p^N)_\infty }
\cdot 
\prod_{1\leq i\leq j\leq N}
{(p^{N-j+i}t x_i/x_j;q,p^N)_\infty \over (p^{N-j+i}qx_i/tx_j;q,p^N)_\infty }
f^{\widehat{\mathfrak gl}_N}(x,p|s,\kappa|q,t)\nonumber\\
=& f^{\widehat{\mathfrak gl}_N}(x,p|s,\kappa|q,q/t). 
\end{align*}
Setting $p=\widetilde{p} t$ and taking the limit $t \rightarrow 0$, we have 
\begin{align}
&\prod_{i=1}^N {1\over (\widetilde{p} q x_{i+1}/x_i;q)_\infty}
\lim_{t\rightarrow 0} f^{\widehat{\mathfrak gl}_N}(x,t \widetilde{p}|s,\kappa|q,t)
=
f^{\widehat{\mathfrak gl}_N{\rm Toda}}(x, \widetilde{p}|s,\kappa|q). \label{toda-Poincare-eq}
\end{align}

\begin{lem}\label{lem-1}
We have
\begin{align*}
&\lim_{t\rightarrow 0}
\prod_{i,j=1}^N
{\Nk^{(j-i|N)}_{\lambda^{(i)},\lambda^{(j)}} (t s_j/s_i|q,\kappa) \over 
\Nk^{(j-i|N)}_{\lambda^{(i)},\lambda^{(j)}} (s_j/s_i|q,\kappa)}
=
\prod_{i,j=1}^N
{1\over 
\Nk^{(j-i|N)}_{\lambda^{(i)},\lambda^{(j)}} (s_j/s_i|q,\kappa)},
\nonumber \\
&\lim_{t\rightarrow 0} \,
(\widetilde{p} t^2/q)^{\sum_i |\lambda^{(i)}|}
\prod_{i,j=1}^N
{\Nk^{(j-i|N)}_{\lambda^{(i)},\lambda^{(j)}} (q s_j/ts_i|q,\kappa) \over 
\Nk^{(j-i|N)}_{\lambda^{(i)},\lambda^{(j)}} (s_j/s_i|q,\kappa)}
\\=&
 \prod_{i=1}^N s_i^{-m_i} q^{-m_i^2/2} (\widetilde{p}/\kappa)^{|\lambda^{(i)}|} \cdot
\prod_{i,j=1}^N
{1 \over 
\Nk^{(j-i|N)}_{\lambda^{(i)},\lambda^{(j)}} (s_j/s_i|q,\kappa)}
\nonumber .
\end{align*}
\end{lem}

\begin{lem}\label{lem-2}
We have
$
\Delta \,\,x^\lambda  \prod_{i=1}^N x_i^{m_i}=\sum_{i=1}^N
\left(
{1\over 2} \lambda_i^2+ \lambda_i m_i +{m_i^2\over 2} \right)x^\lambda  \prod_{i=1}^N x_i^{m_i},
$
namely
$
q^\Delta \,\,x^\lambda  \prod_{i=1}^N x_i^{m_i}=
q^{{1\over 2} \sum_i \lambda_i^2}
\prod_{i=1}^N s_i^{m_i} q^{m_i^2/2} \cdot x^\lambda  \prod_{i=1}^N x_i^{m_i}.
$
\end{lem}

\begin{lem}\label{lem-3}
{}From Lemmas \ref{lem-1} and \ref{lem-2}, we have
\begin{align*}
&q^{{1\over 2} \sum_i \lambda_i^2}\,\,
x^\lambda
\lim_{t\rightarrow 0} f^{\widehat{\mathfrak gl}_N}(x,t \widetilde{p}|s,\kappa|q,t)\nonumber\\
=&\,\,
q^\Delta  T_{\kappa,\widetilde{p}} \,\,x^\lambda
\lim_{t\rightarrow 0} f^{\widehat{\mathfrak gl}_N}(x,t \widetilde{p}|s,\kappa|q,t/q)
=q^\Delta  T_{\kappa,\widetilde{p}} \,\,x^\lambda f^{\widehat{\mathfrak gl}_N{\rm Toda}}(x, \widetilde{p}|s,\kappa|q).
\end{align*}
\end{lem}

\noindent
{\it Proof of Proposition \ref{non-stationary-Toda-2}.}
The formula (\ref{non-stationary-Toda-f-ex}) follows from Lemma \ref{lem-1}.
\qed

\noindent
{\it Proof of Proposition \ref{non-stationary-Toda}.}
Lemma \ref{lem-3} shows that (\ref{toda-Poincare-eq}), namely 
the Poincar\'e duality conjecture in $t\rightarrow 0$, is
nothing but the non-stationary Toda equation (\ref{non-stationary-Toda-eq}).
\qed

\subsection{Limit from non-stationary affine $q$-Toda to stationary affine $q$-Toda: the limit $\kappa\rightarrow 1$. }
Even though the operator ${\mathcal T}^{\widehat{\mathfrak gl}_N}(\kappa)$ itself has no singularity at $\kappa=1$, 
and the commutativity 
$[{\mathcal T}^{\widehat{\mathfrak gl}_N}(1),D^{\widehat{\mathfrak gl}_N{\rm Toda}}]=0$ takes place, 
the behavior of $f^{\widehat{\mathfrak gl}_N{\rm Toda}}$ in the vicinity of $\kappa=1$ is quite subtle. 
This is because the operator ${\mathcal T}^{\widehat{\mathfrak gl}_N}(1)$ commutes with multiplication by 
any function in $\widetilde{p}$, producing drastic degenerations in the spectrum, resulting 
Jordan blocks of infinite size lacking any eigenspaces.
However, we still find some nice structure.

\begin{dfn}
Let $\alpha(\widetilde{p})=\alpha(\widetilde{p}|s,\kappa|q)=\sum_{d\geq 0} \widetilde{p}^{Nd} 
\alpha_{d}(s,\kappa|q)$ be the constant term of $f^{\widehat{\mathfrak gl}_N{\rm Toda}}(x, \widetilde{p}|s,\kappa|q)$ with respect to $x_i$'s.
Namely,
\begin{align}
\alpha(\widetilde{p})
=&
\sum_{
\lambda^{(1)},\ldots,\lambda^{(N)}\in {\mathsf P}\atop 
m_1=\cdots= m_N=0}
(\widetilde{p}/\kappa)^{|{\boldsymbol \lambda}|}
\prod_{i,j=1}^N
{1 \over 
\Nk^{(j-i|N)}_{\lambda^{(i)},\lambda^{(j)}} (s_j/s_i|q,\kappa)}.
\nonumber
\end{align}

\end{dfn}

\begin{con}\label{analyticity-Toda}
We have the properties:
\begin{enumerate}
\item
The series $f^{\widehat{\mathfrak gl}_N{\rm Toda}}$ is convergent on a certain domain.
With respect to $\kappa$, it is regular on a certain punctured disk 
$\{\kappa \in {\mathbb C}| |\kappa-1| <r,\kappa\neq 1\}$ .
\item
The $f^{\widehat{\mathfrak gl}_N{\rm Toda}}$ and $\alpha(\widetilde{p})$ are essential singular at $\kappa=1$. 
The $\alpha_{d}(s,\kappa|q)$ has a pole of degree $d$ in $\kappa$ at $\kappa=1$.
\item
The ratio $f^{\widehat{\mathfrak gl}_N{\rm Toda}}/\alpha(\widetilde{p})$ is regular at $\kappa=1$. 
\item
The ratio 
${\alpha(\widetilde{p}) / \alpha(\kappa \widetilde{p})}$ is regular at $\kappa=1$. 
The limit $\lim_{\kappa\rightarrow 1}{\alpha(\widetilde{p}) / \alpha(\kappa \widetilde{p})}$  is a nontrivial function, 
which is Taylor expanded in $\widetilde{p}^N$. 
\end{enumerate}
\end{con}

Conjecture \ref{analyticity-Toda} suggests the following scheme 
for analyzing the eigenvalue problem of stationary $q$-affine Toda equation. 
\begin{dfn}
Assuming Conjecture \ref{analyticity-Toda}, set
\begin{align*}
&
f^{{\rm st.}\,\widehat{\mathfrak gl}_N{\rm Toda}}(x, \widetilde{p}|s|q)
={f^{\widehat{\mathfrak gl}_N{\rm Toda}}(x,\widetilde{p}|s,\kappa|q) \over \alpha(\widetilde{p}|s,\kappa|q) } \Biggl|_{\kappa=1},
\quad 
\varepsilon(\widetilde{p}|s|q)= q^{{1\over 2} \sum_{i=1}^n \lambda_i^2} 
{\alpha(\widetilde{p}|s,\kappa|q) \over \alpha(\kappa \widetilde{p}|s,\kappa|q)}\Biggl|_{\kappa=1}.
\end{align*}
\end{dfn}
{} The conjecture (\ref{non-stationary-Toda-eq}) implies that 
we have
\begin{align*}
{\mathcal T}^{\widehat{\mathfrak gl}_N}(\kappa) \,
x^\lambda {f^{\widehat{\mathfrak gl}_N{\rm Toda}}(x,\widetilde{p}|s,\kappa|q) \over \alpha(\widetilde{p}|s,\kappa|q)} =
q^{{1\over 2} \sum_{i=1}^n \lambda_i^2} {\alpha(\widetilde{p}|s,\kappa|q)\over \alpha(\kappa\widetilde{p}|s,\kappa|q)}\,
x^\lambda {f^{\widehat{\mathfrak gl}_N{\rm Toda}}(x,\widetilde{p}|s,\kappa|q) \over \alpha(\widetilde{p}|s,\kappa|q)}.
\end{align*}
Then Conjecture \ref{analyticity-Toda} asserts the existence of the limit $\kappa\rightarrow 1$ of this.
\begin{con}
We have
\begin{align*}
{\mathcal T}^{\widehat{\mathfrak gl}_N}(1)\, x^\lambda  f^{{\rm st.}\,\widehat{\mathfrak gl}_N{\rm Toda}}(x,\widetilde{p}|s|q)
 =\varepsilon(\widetilde{p}|s|q) \,
x^\lambda  f^{{\rm st.}\widehat{\mathfrak gl}_N{\rm Toda}} (x,\widetilde{p}|s|q).
\end{align*}
\end{con}

Because of the commutativity 
$[{\mathcal T}^{\widehat{\mathfrak gl}_N}(1),D^{\widehat{\mathfrak gl}_N{\rm Toda}}]=0$,
the $x^\lambda  f^{{\rm st.}\,\widehat{\mathfrak gl}_N{\rm Toda}}(x,\widetilde{p}|s|q)$ 
should also be the eigenfunction of the stationary affine Toda operator 
$D^{\widehat{\mathfrak gl}_N{\rm Toda}}$.
\begin{con}
We have
\begin{align*}
&D^{\widehat{\mathfrak gl}_N{\rm Toda}}\,x^\lambda  f^{{\rm st.}\,\widehat{\mathfrak gl}_N{\rm Toda}}(x,\widetilde{p}|s|q)
 =\varepsilon^D(\widetilde{p}|s|q) \,
x^\lambda  f^{{\rm st.}\,\widehat{\mathfrak gl}_N{\rm Toda}}(x,\widetilde{p}|s|q),\\
&
\varepsilon^D(\widetilde{p}|s|q)  =\sum_{i=1}^n s_i+\sum_{k=1}^\infty  \varepsilon^{D}_k(s|q) \widetilde{p}^{Nk}.
\end{align*}
\end{con}

\section{Non-stationary and stationary elliptic Calogero-Sutherland equation: limt $q,t\rightarrow 1$}\label{nonstationary-eCS-system}

\subsection{Non-stationary elliptic Calogero-Sutherland equation}
Let $H^{\rm eCS}(p)$ be the elliptic Calogero-Sutherland operator in Definition \ref{H-eCS}.
The non-stationary elliptic Calogeor equation is defined to be the eigenvalue equation
\begin{align*}
\Bigl( k \,p{\partial \over \partial p} +H^{\rm eCS}(p|\beta) \Bigr)\psi(x,p|\lambda,k|\beta)=
\varepsilon(\lambda)\,\,\psi(x,p|\lambda,k|\beta). 
\end{align*}

\begin{dfn}
Set
$q=e^h, t=e^{\beta h}, s_i=q^{\lambda_i}$ ($1\leq i\leq N$), and $\kappa=e^{k h}$.
Define $f^{\rm eCS}(x,p|\lambda,k|\beta)$ and $\varphi^{\rm eCS}(x,p|\lambda,k|\beta)$ by the limits
\begin{align*}
&f^{\rm eCS}(x,p|\lambda,k|\beta)=\lim_{h\rightarrow 0} f^{{\mathfrak gl}_N}(x,p|s,\kappa|q,q/t),\\
&
\varphi^{\rm eCS}(x,p|\lambda,k|\beta)=\lim_{h\rightarrow 0} \varphi^{{\mathfrak gl}_N}(x,p|s,\kappa|q,q/t).
\end{align*}
Note that $\varphi^{\rm eCS}(x,p|\lambda,k|\beta)=\psi_0(x,p|\beta)f^{\rm eCS}(x,p|\lambda,k|\beta)$.
\end{dfn}

\begin{con}\label{nonstationary-eCS}
We have
\begin{align*}
\Bigl( k \,p{\partial \over \partial p} +H^{\rm eCS}(p|\beta) \Bigr)x^\lambda\varphi^{\rm eCS}(x,p|\lambda,k|\beta)=
{1\over 2}\sum_{i=1}^N \lambda_i^2\,\,x^\lambda\varphi^{\rm eCS}(x,p|\lambda,k|\beta). 
\end{align*}
Or equivalently that
\begin{align*}
&\left( k \left({1\over  \psi_0}
p {\partial \psi_0\over \partial p} \right)  +k \,p{\partial \over \partial p} +H_{\beta}(p)+
{ \beta(\beta-1)N \over 2}V_0(p^N) \right)x^\lambda f^{\rm eCS}(x,p|\lambda,k|\beta) \\
=&
{1\over 2}\sum_{i=1}^N \lambda_i^2\,\,x^\lambda f^{\rm eCS}(x,p|\lambda,k|\beta). \nonumber
\end{align*}
\end{con}

\subsection{Explicit formula for $f^{\rm eCS}(x,p|\lambda,k|\beta)$}
Denote by $(a)_n=a(a+1)\cdots (a+n-1)$ the ordinary shifted product.

\begin{dfn}
For $l\in \mathbb{ Z}/N\mathbb{Z}$, and $\lambda,\mu \in {\mathsf P}$, set
\begin{align*}
\Nk^{(l|N)}_{\lambda,\mu}(v|k)
=&
 \prod_{j\geq i\geq 1 \atop j-i \equiv l \,\,({\rm mod}\,N)}
(v-\mu_i+\lambda_{j+1}-(i-j)k )_{\lambda_j-\lambda_{j+1}}\\
&\times
\prod_{\beta\geq \alpha \geq 1  \atop \beta-\alpha \equiv -l-1 \,\,({\rm mod}\,N)}
(v +\lambda_{\alpha}-\mu_\beta +(\alpha-\beta-1)k)_{\mu_{\beta}-\mu_{\beta+1}}. \nonumber
\end{align*}
\end{dfn}

\begin{lem}
Let $q=e^h, t=e^{\beta h}, s_i=q^{\lambda_i}$ ($1\leq i\leq N$), and $\kappa=e^{k h}$.
We have
\begin{align*}
& f^{\rm eCS}(x,p|\lambda,k|\beta)=\lim_{h\rightarrow 0}f^{\widehat{\mathfrak gl}_N}(x,p|e^{h\lambda} ,e^{hk}|e^h,e^{h(1-\beta)})\\
=&
\sum_{\mu^{(1)},\ldots,\mu^{(N)}\in {\mathsf P}}
\prod_{i,j=1}^N
{\Nk^{(j-i|N)}_{\mu^{(i)},\mu^{(j)}} (1-\beta+\lambda_j-\lambda_i |k) \over \Nk^{(j-i|N)}_{\mu^{(i)},\mu^{(j)}} (\lambda_j-\lambda_i |k)}
\cdot \prod_{\beta=1}^N\prod_{\alpha\geq 1} ( p x_{\alpha+\beta}/x_{\alpha+\beta-1})^{\mu^{(\beta)}_\alpha}.\nonumber
\end{align*}
\end{lem}

\subsection{$\widehat{\mathfrak sl}_N$ dominant integrable character case}
Let $K,\mu$ be as in Definition \ref{dominant-integrable-weights}. This in the differential case means that we set $k=-{K\over N}-1,\beta=1$.
Note that we have the $\widehat{\mathfrak sl}_N$  character as
$x^\mu f^{\rm eCS}(x,p|\mu,-(K+N)/N|1)={\rm ch}^{\widehat{\mathfrak sl}_N}_{L(\Lambda(K,\mu))}/{(p^N;p^N)_\infty}$. 
Note also that when $\beta=1$, the quasi-ground state $ \psi_0(x,p|\beta=1)$ 
(given in Definition \ref{quasi-ground-state}) divided by ${(p^N;p^N)_\infty}$, {\it i.e.} $ \psi_0(x,p|1) /{(p^N;p^N)_\infty}$, 
is nothing but
Weyl-Kac's denominator for the dominant integrable characters of $\widehat{\mathfrak sl}_N$.  
Then the case $\beta=1$ of Conjecture \ref{nonstationary-eCS} reduces to the heat equation for the numerator of the Weyl-Kac formula, {\it i.e.} for
$ \psi_0(x,p|1) \times {\rm ch}^{\widehat{\mathfrak sl}_N}_{L(\Lambda(K,\mu))}/{(p^N;p^N)_\infty}=
x^\mu \varphi^{\rm eCS}(x,p|\mu,-(K+N)/N|1)$.

\begin{prp} Let $\beta=1$, and $K,\mu$ be as in Definition \ref{dominant-integrable-weights}.
We have the heat equation 
\begin{align*}
&\Bigl( -{K+N\over N} \,p{\partial \over \partial p} +{1\over 2} \Delta \Bigr)x^\mu \varphi^{\rm eCS}(x,p|\mu,-(K+N)/N|1)\nonumber\\
&\qquad ={1\over 2}\sum_{i=1}^N \mu_i^2\,\,x^\mu \varphi^{\rm eCS}(x,p|\mu,-(K+N)/N|1). 
\end{align*}
\end{prp}

Next, consider the case $K=0$ ({\it i.e.} $k=-\beta$), and $\mu=\emptyset$ . Note that we have 
$ f^{\rm eCS}(x,p|0,\ldots,0,-N\beta|\beta)={1/ (p^N,p^N)_\infty}$. 
This case gives another nontrivial check (for arbitrary $\beta$) of Conjecture \ref{nonstationary-eCS}.
\begin{prp}
We have
\begin{align}
\left(-\beta p{\partial \over \partial p}+ H^{\rm eCS}(p|\beta) \right)\psi_0(x,p|\beta) {1\over (p^N,p^N)_\infty}=0. \label{character}
\end{align}
\end{prp}

We remark that this is a special case of Langmann's kernel function identity \cite{L}.

\begin{proof}
Equation (\ref{character}) is equivalent to the following $\beta$ independent equation
\begin{align*}
& \sum_{1\leq i<j\leq N}
{1\over  \Theta_{p^N}(p^{j-i}x_j/x_i)}
p{\partial\over \partial p} \Theta_{p^N}(p^{j-i}x_j/x_i)\nonumber\\
=&
 \sum_{1\leq i<j\leq N} \left(
 {N\over 2} { \Theta^{(2)}_{p^N} (p^{j-i}x_j/x_i)\over \Theta_{p^N}(p^{j-i}x_j/x_i)}-
 {N-2j+2i\over 2} {\Theta^{(1)}_{p^N} (p^{j-i}x_j/x_i)\over \Theta_{p^N}(p^{j-i}x_j/x_i)}\right),
\end{align*}
which follows from the heat equations for the theta function for $1\leq i<j\leq N$
\begin{align*}
- p{\partial\over \partial p} \Theta_{p^N} (p^{j-i}x_j/x_i)+
{N\over 2} \Theta^{(2)}_{p^N} (p^{j-i}x_j/x_i)-
 {N-2j+2i\over 2}\Theta^{(1)}_{p^N} (p^{j-i}x_j/x_i)=0.
\end{align*}
\end{proof}

\subsection{Stationary elliptic Calogero-Sutherland equation: case $k=0$}

\begin{dfn}
Let 
$\alpha(p|\lambda,k|\beta)=\sum_{d\geq 0} p^{Nd} 
\alpha_d(p|\lambda,k|\beta)$ be the constant term of the series $ f^{\rm eCS}(x,p|\lambda,k|\beta)$ with respect to $x_i$'s.
Namely,
\begin{align}
\alpha(p|\lambda,k|\beta)
=&
\sum_{
\lambda^{(1)},\ldots,\lambda^{(N)}\in {\mathsf P}\atop 
m_1=\cdots=m_N=0}
p^{|{\boldsymbol \lambda}|}
\prod_{i,j=1}^N
{\Nk^{(j-i|N)}_{\lambda^{(i)},\lambda^{(j)}} (1-\beta+\mu_j-\mu_i |k) \over \Nk^{(j-i|N)}_{\lambda^{(i)},\lambda^{(j)}} (\mu_j-\mu_i |k)}.
\nonumber
\end{align}

\end{dfn}

\begin{con}\label{analyticity-eCS}
We have the properties:
\begin{enumerate}
\item
The series $ f^{\rm eCS}$ is convergent on a certain domain.
With respect to $k$, it is regular on a certain punctured disk 
$\{k \in {\mathbb C}| |k| <r,k\neq 0\}$ .
\item
The $f^{\rm eCS}$ and $\alpha(p|\lambda,k|\beta)$ are essential singular at $k=0$. 
The $\alpha_d(p|\lambda,k|\beta)$ has a pole of degree $d$ in $k$ at $k=0$.
\item
The ratio $f^{\rm eCS}/\alpha(p|\lambda,k|\beta)$ is regular at $k=0$. 
\item
The derivative 
${1\over k} p {\partial /\partial p} \log \alpha(p|\lambda,k|\beta)$ is regular at $k=0$, 
the limit $k\rightarrow 0$ of which being a nontrivial Taylor series in $p^N$. 
\end{enumerate}
\end{con}
\begin{dfn}
Assuming Conjecture \ref{analyticity-eCS}, set
\begin{align*}
& \varphi^{\rm st.\,eCS}(x,p|\lambda,k|\beta)={ \varphi^{\rm eCS}(x,p|\lambda,k|\beta)\over \alpha(p|\lambda,k|\beta)} \Bigl|_{k=0},
\,\,
\varepsilon(p|\lambda|\beta)= {1\over 2} \sum_{i=1}^n \lambda_i^2+
{1\over k} p {\partial \over\partial p} \log \alpha(p|\lambda,k|\beta)\Bigl|_{k=0}.
\end{align*}
\end{dfn}

Conjectures \ref{nonstationary-eCS} and \ref{analyticity-eCS} assert the following conjecture for the 
stationary elliptic Calogero-Sutherland equation.
\begin{con}
We have 
$
H^{\rm eCS}(p)\,x^\lambda\varphi^{\rm st.\,eCS}(x,p|\lambda|\beta)=
\varepsilon(p|\lambda|\beta)\,x^\lambda\varphi^{\rm st.\,eCS}(x,p|\lambda|\beta).
$
\end{con}


\end{document}